\documentclass[12pt]{amsart}

\usepackage{graphicx}
\usepackage{amsfonts}
\usepackage{epsf}
\usepackage{amssymb}
\usepackage{amsmath}
\usepackage{amscd}
\usepackage{hyperref}

\newtheorem{theorem}{Theorem}[section]
\newtheorem{proposition}[theorem]{Proposition}
\newtheorem{lemma}[theorem]{Lemma}

\newtheorem{corollary}[theorem]{Corollary}

\theoremstyle{remark} 
\newtheorem{definition}[theorem]{Definition}
\newtheorem{example}[theorem]{Example}
\newtheorem{remark}[theorem]{Remark} 

\newcommand{\kerr}{\mbox{Ker} }

\newcommand{\s}{\mathfrak{s}}
\renewcommand{\t}{\mathfrak{t}}

\newcommand{\gq}{g_{\mathbb Q}}
\newcommand{\gz}{g_{\mathbb Z}}

\input xy
\xyoption{all}

\begin{document}

\title{Heegaard Floer genus bounds for Dehn surgeries on knots}
\begin{abstract}
We provide a new obstruction for a rational homology $3$-sphere to arise by Dehn surgery on a given knot in the $3$-sphere. The obstruction takes the form of an inequality involving the genus of the knot, the surgery coefficient, and a count of {\em $L$-structures} on the $3$-manifold, that is spin$^c$-structures with the simplest possible associated Heegaard Floer group. Applications include an obstruction for two framed knots to yield the same $3$-manifold, an obstruction that is particularly effective when working with families of framed knots. We introduce the {\em rational and integral Dehn surgery genera} for a rational homology $3$-sphere, and use our inequality to provide bounds, and in some cases exact values, for these genera. We also demonstrate that the difference between the integral and rational Dehn surgery genera can be arbitrarily large. 
\end{abstract}
\author{Stanislav Jabuka}
\email{jabuka@unr.edu}
\address{Department of Mathematics and Statistics, University of Nevada, Reno NV 89557.}
\thanks{The author was partially supported by grant \#246123 from the Simons Foundation, and from a research grant from the University of Nevada, Reno.}
\maketitle
\section{Introduction}
\subsection{Preface} 
It is well known \cite{Lickorish, Wallace} that every oriented, closed $3$-manifold can be constructed via Dehn surgery on a framed link $L$ in $S^3$. The framed link $L$ in this construction is highly non-unique, but any two framed links yielding the same $3$-manifold are related by a finite sequence of blow-ups, blow-downs, handle slides and isotopies \cite{Kirby} (two such framed links shall be called {\em surgery equivalent}). While in theory this curbs the non-uniqueness, in practice it is often not easy to tell if two framed links are related in this manner. Indeed, even in the simpler case of framed knots, it remains a challenge. The first example of an integral homology sphere that can be obtained by surgeries on two different knots was found by Lickorish \cite{Lickorish2}, and many examples have followed since then \cite{Brakes, Kawauchi, Livingston, Soma, Teragaito}. 

To help restrain this many-to-one phenomenon, we derive an obstruction for a $3$-manifold $Y$ to be the result of a $p/q$-framed surgery on a knot $K\subset S^3$. The obstruction takes the form of an inequality (Theorem \ref{main}) involving $p,q$, the genus of the knot $K$ and data derived from the Heegaard Floer homology groups of $Y$. 

For a given framed knot, this inequality bounds from below the genus of any surgery equivalent framed knot.  

Among framed links in $S^3$, those with integer framings play a special role. Indeed, any such link $L$ does not only yield $3$-manifold $Y$ via Dehn surgery, but also describes a smooth, oriented $4$-manifold $X$ with $\partial X = Y$, obtained by attaching $4$-dimensional $2$-handles to the $4$-ball $D^4$, attached to the link $L\subset \partial D^4$. For this reason, we shall heed special attention to integral surgeries when stating our results. 
\subsection{Definitions}
If $r=p/q$ is a rational number in lowest terms, we shall write $S^3_r(K)$ or $S^3_{p/q}(K)$ to denote the $3$-manifold resulting from $r$-framed Dehn surgery on the knot $K\subset S^3$. The Seifert genus of a knot $K$ shall be denoted $g(K)$. 

\begin{definition} Let $Y$ be a rational homology $3$-sphere. We define its 
{\em rational and integral Dehn surgery genera $g_\mathbb Q(Y)$ and $g_\mathbb Z(Y)$} as:
\begin{align} \nonumber
\gq (Y)& = \left\{
\begin{array}{cl}
 \min \left\{ g(K) \, \big| \,\, \,  \text{$Y=S^3_r(K)$, $r\in \mathbb Q$.} \right\}& ; \quad \text{If $Y=S^3_{r}(K)$ for some $K$.} \cr 
 \infty &;\quad \text{Otherwise.}
 \end{array}
 \right. \cr &\cr
\gz (Y) & = \left\{
\begin{array}{cl}
 \min \left\{ g(K) \, \big| \,\, \,  \text{$Y=S^3_p(K)$, $p\in \mathbb Z$.} \right\}& ; \quad \text{If $Y=S^3_{p}(K)$ for some $K$.} \cr 
 \infty &;\quad \text{Otherwise.}
 \end{array}
 \right.
\end{align}
\end{definition}
\noindent Note that $\gq(Y)=0$ if and only if $Y$ is a lens space. 

For a closed and oriented $3$-manifold $Y$, let $Spin^c(Y)$ denote its affine space of spin$^c$-structures and let $\widehat{HF}(Y,\s)$ be its associated {\em hat} version of the Heegaard Floer homology group (these are defined in Section \ref{HeegaardFloerGroupsDoneRight}). 
\begin{definition}
Let $Y$ be a rational homology 3-sphere. A spin$^c$-structure $\s\in Spin^c(Y)$ is called an {\em $L$-structure} if  $\widehat{HF}(Y,\s) \cong \mathbb{Z}$. We shall write $\ell(Y)$ or simply $\ell$ to denote the number of $L$-structures on $Y$.  
\end{definition}

Our use of nomenclature follows that of \cite{OzsvathSzabo12} where a rational homology sphere, all of whose spin$^c$-structures are $L$-structures, is called an {\em $L$-space}. 
\subsection{Results}
With these definitions in place, we turn to our surgery obstruction. 
\begin{theorem} \label{main}
Let $Y$ be a rational homology sphere with $|H_1(Y;\mathbb{Z})|=p$. If $Y$ is obtained by $p/q$-surgery 
on a knot $K\subset S^3$ with $g(K)\ge 1$, then  
\begin{equation} \label{MainInequality}
2g(K)-1 \ge \frac{p-\ell}{|q|} .
\end{equation}
Here $\ell$ is the number of $L$-structures on $Y$.
\end{theorem}
Using different approaches, other genus bounds stemming from Heegaard Floer homology for knots with prescribed surgeries have been obtained by Ozs\'ath and Szab\'o \cite{OzsvathSzabo12} (providing four-ball genus bounds for knots with lens space surgeries), Rasmussen \cite{Rasmussen2} (showing that if a surgery of slope $p$ on a genus $g$ knot yields a lens space, then $p\le 4g+3$)  and Greene \cite{Greene1} (demonstrating the inequality $2g-1\le p - \sqrt{3p+1}$ for a knot $K$ of genus $g$ on which integral $p>0$ surgery yields a lens space that bounds a sharp $4$-manifold with torsion-free first homology). 
\begin{corollary} \label{main2}
Let $Y$ be a rational homology sphere different from a lens space, and let $\ell$ be the number of $L$-structures on $Y$. Then 
\begin{equation}\label{MainInequalityCorollary}
 2\gz (Y)-1 \ge |H_1(Y;\mathbb{Z})| - \ell.
 \end{equation}
\end{corollary}
%
%
%
\begin{remark}
Inequality \eqref{MainInequality} from Theorem \ref{main} unfortunately becomes vacuous for $L$-spaces and integral homology spheres. In both cases the inequality reduces to $g(K)\ge 1$ which is a hypothesis of the theorem.
\end{remark}
\subsection{Examples} We provide families of examples to illustrate two points: 
\begin{itemize}
\item[(a)] Inequality \eqref{MainInequality} from Theorem \ref{main} is sharp for infinitely many surgeries (Proposition \ref{PropositionAboutTau} and Example \ref{ExampleSharpness}).
\item[(b)]  Theorem \ref{main} can be used to provide infinitely many examples of $3$-manifolds $Y$ for which $\gz (Y)>\gq(Y)$. Indeed, the difference $\gq(Y)-\gz(Y)$ can be made arbitrarily large, while being finite (Example \ref{ExampleDiscrepancy}).  
\end{itemize}
For a knot $K$ in $S^3$, let $\tau(K)$ denote its Ozsv\'ath-Szab\'o tau invariant \cite{OzsvathSzabo11} (see Section \ref{SectionTheTauInvariant} for a detailed definition). 
\begin{proposition} \label{PropositionAboutTau}
Let $K\subset S^3$ be a knot with $|\tau(K)| = g(K)>0$ and let $p,q$ be a pair of positive, relatively prime integers with $p-(2g(K)-1)q>0$. Then 
$$\ell\left(S^3_{\varepsilon \frac{p}{q}}(K) \right) =p-(2g(K)-1)q,$$
where $\varepsilon = -Sign (\tau(K))$. 
\end{proposition}
Any knot $K$ as in Proposition \ref{PropositionAboutTau} renders inequality \eqref{MainInequality} sharp. Explicit examples of such knots are provided by  $L$-knots (knots which yield an $L$-space by some positive, integral surgery \cite{OzsvathSzabo12}, for instance torus knots $T_{(a,b)}$ with $ab>0$) and their mirrors, and alternating knots $K$ with signature $\sigma (K) = \pm 2g(K)$.  
\begin{example} \label{ExampleSharpness}
Let $K$ be a knot meeting the hypothesis of Proposition \ref{PropositionAboutTau} and set $\varepsilon = -Sign (\tau(K))$. Then, for any positive integer $p> 2g(K)-1$, one obtains
$$\gz \left( S^3_{\varepsilon p}(K)\right)=g(K).$$
For instance, taking a positive integer $g$ and letting $K$ be the torus knot $T_{(2,2g+1)}$, one obtains $\gz \left( S^3_{-p} (T_{(2,2g+1)})\right) = g$ (still with $p>2g-1$). 
\end{example}
Computations justifying our claims in the next example are deferred to Section \ref{ExamplesSection}. 
\begin{example} \label{ExampleDiscrepancy}
We exhibit an infinite family of rational homology $3$-spheres $Y_n$ for which $\gz(Y_n)-\gq(Y_n) \ge \frac{n-1}{2}$. Namely, for $n\in \mathbb N$ let $Y_n$ be the result of $-\frac{4n+1}{n}$-framed surgery on the Figure Eight knot. Then $\ell(Y_n) = 3n+1$ so that $Y_n$ is not an $L$-space for any choice of $n$. Since the genus of the Figure Eight knot is $1$, it follows that $\gq(Y_n) = 1$. Inequality \eqref{MainInequalityCorollary} shows that $\gz(Y_n) \ge \frac{n+1}{2}$ leading to $\gz(Y_n) - \gq(Y_n) \ge \frac{n-1}{2}$. In Section \ref{ExamplesSection} we show that $Y_n$ also arises as an integral surgery on a knot showing $\gz(Y_n)-\gq(Y_n)$ to be finite. 
\end{example}
\subsection{Applications} 
As already alluded to in the introduction, Theorem \ref{main} can obstruct surgery equivalence among framed knots.  We remark that we are only using the ranks of the Heegaard Floer groups for this obstruction. In another direction, the Heegaard Floer correction terms can also be used to furnish surgery obstructions, see for instance \cite{Doig}.

For a pair of framed knots $(K_1,\frac{p}{q_1})$ and $(K_2,\frac{p}{q_2})$, the obstruction is evaluated by computing the number $\ell$ of $L$-structures on $Y=S^3_{p/q_1}(K_1)$, and by asking whether the inequality 
$$2g(K_2)-1 \ge \frac{|p|-\ell}{|q_2|}$$
is violated. If the answer is \lq Yes\rq, then  $(K_1,\frac{p}{q_1})$ and $(K_2,\frac{p}{q_2})$ are not surgery equivalent. 

A comparison of the Heegaard Floer homology groups for $S^3_{p/q_1}(K_1)$, $S^3_{p/q_2}(K_2)$ is of course a stronger obstruction to surgery equivalence, but it also involves more computation. This becomes especially prominent when $K_2$ is not fixed but allowed to vary across a family of knots. In such a happenstance, Theorem \ref{main} can be used as a significant shortcut to ruling out surgery equivalence. We illustrate this point with two examples. 
\begin{example} \label{ExampleOfLMinimizingKnots}
Consider a pair of surgery equivalent framed knots $(K_1,-\frac{p}{q_1})$ and $(K_2,-\frac{p}{q_2})$ with $p,q_i>0$, $\gcd (p,q_i)=1$ and $p-(2g(K_i)-1)q_i>0$.  
\begin{itemize}
\item[(i)] If $\tau(K_1)=g(K_1)$  then 
$$ 2g(K_2)-1 \ge \frac{q_1}{q_2} \cdot (2g(K_1)-1). $$
\item[(ii)] If $\tau(K_i)=g(K_i)$ for $i=1,2$
$$ 2g(K_2)-1 = \frac{q_1}{q_2} \cdot (2g(K_1)-1). $$
\end{itemize}
We are fixing the knot $K_1$  and allowing $K_2$ to vary through the family of all knots in $S^3$ (in part (i)) or through the family of knots with $\tau(K_2)=g(K_2)$ (in part (ii)). In each case, an application of Theorem \ref{main} gives considerable restrictions on the genera and framings involved. For instance, if $K_1$ and $K_2$ in case (ii) above are of equal Seifert genus $g$, then the surgery equivalence of $(K_1,-\frac{p}{q_1})$ and $(K_2,-\frac{p}{q_2})$ implies that  $q_1=q_2$. 
\end{example}
\begin{example} \label{ExampleOfNonLMinimizingKnots}
For positive integers $m,k$, let $K_{2m,2k+1}$ be the knot in Figure \ref{pic3}. It is easy to check that $g(K_{2m,2k+1}) = m$. In this example we apply Theorem \ref{main} to give a partial answer to the question: {\em When are the framed knots $(K_{2m,2k+1},\frac{p}{q_1})$ and $(K_{2n,2j+1},\frac{p}{q_2})$ surgery equivalent?} 

Assume that  $p,q_i, p-(2m-1)q_1, p-(2n-1)q_2$ are all positive. We will show in Section \ref{ExamplesSection} that 
$$ \ell = \ell \left( S^3_{-p/q_1}(K_{2m,2k+1}) \right) = p-mq_1.$$
Theorem \ref{main} gives the restriction 
\begin{equation}\label{ObstructionForExampleFour}
 \frac{q_2}{q_1} \ge \frac{m}{2n-1} ,
 \end{equation}
for any framed knot $(K_{2n,2j+1},-\frac{p}{q_2})$ surgery equivalent to $K(_{2m,2k+1},-\frac{p}{q_1})$. 

How good an obstruction is this? In Section \ref{ExamplesSection} we will demonstrate that with the choices of $p=4mn-1$, $q_1=n$, $q_2=m$ and $j=k$, the framed knot $(K_{2m,2k+1},-\frac{4mn-1}{n})$ is surgery equivalent to  $(K_{2n,2k+1},-\frac{4mn-1}{q_2})$ with $q_2=m$. Inequality \eqref{ObstructionForExampleFour} for an indeterminate $q_2$ becomes $q_2\ge mn/(2n-1)$ and is sharp for $n=1$. For values of $n>1$, we are not aware of values of $q_2$ with $\frac{mn}{2n-1} \le |q_2|<m$ that yield the desired surgery equivalence.
\begin{figure}[htb!] 
\centering
\includegraphics[width=10cm]{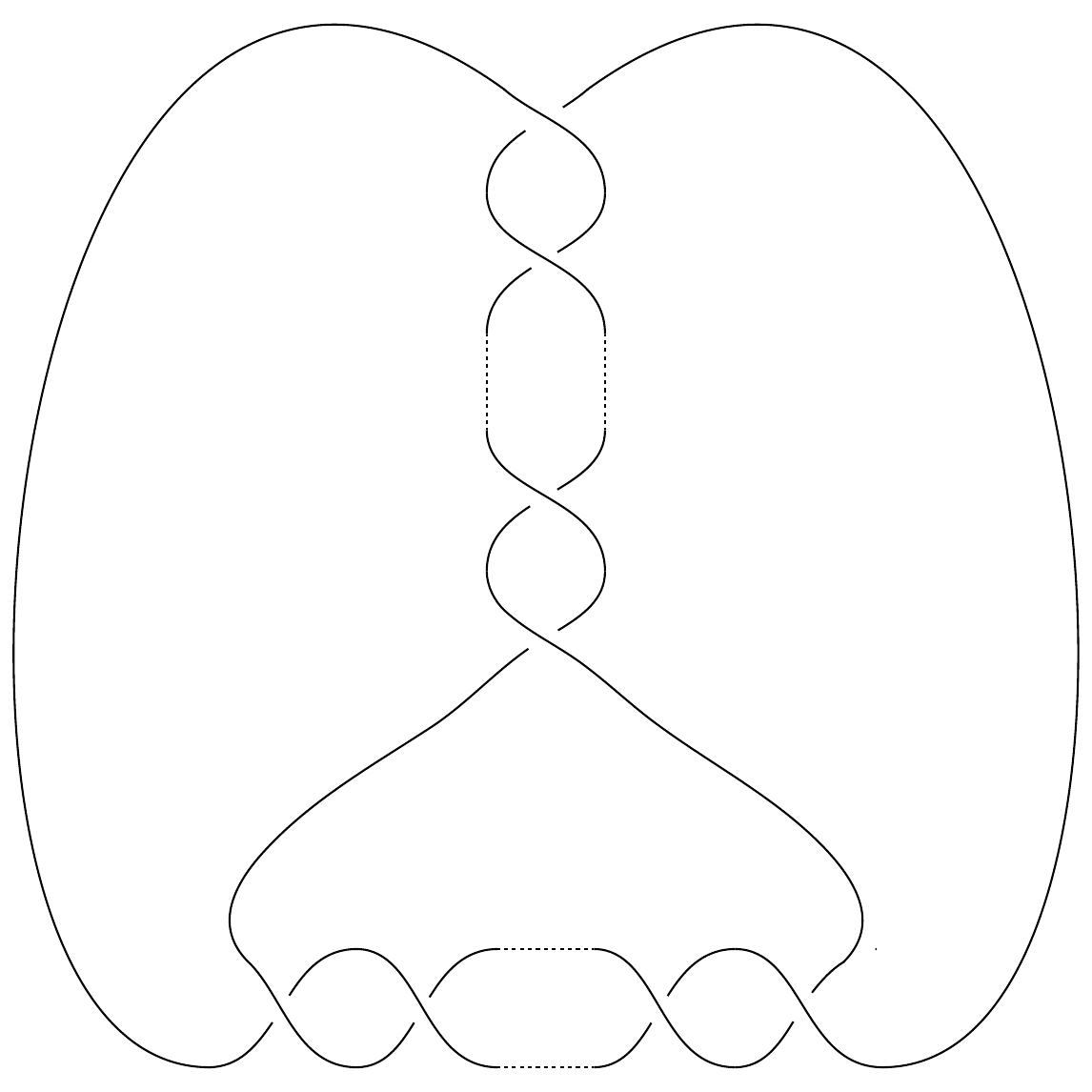}
\put(-215,0){$\underbrace{\phantom{iiiiiiiiiiiiiiiiiiiiiiiiiiiiiiiiiii}}$}
\put(-200,-15){\tiny $2k+1$ right-handed half twists.}
\put(-135,182){$\left. \begin{array}{c} \cr \cr\cr \cr\cr  \cr \cr \cr \cr \cr \end{array}\right\}$}
\put(-110,188){\tiny $2m$ left-handed}
\put(-110,178){\tiny half twists.}
\caption{The knot $K_{2m,2k+1}$ with $m,k \in \mathbb N$. }  \label{pic3}
\end{figure}
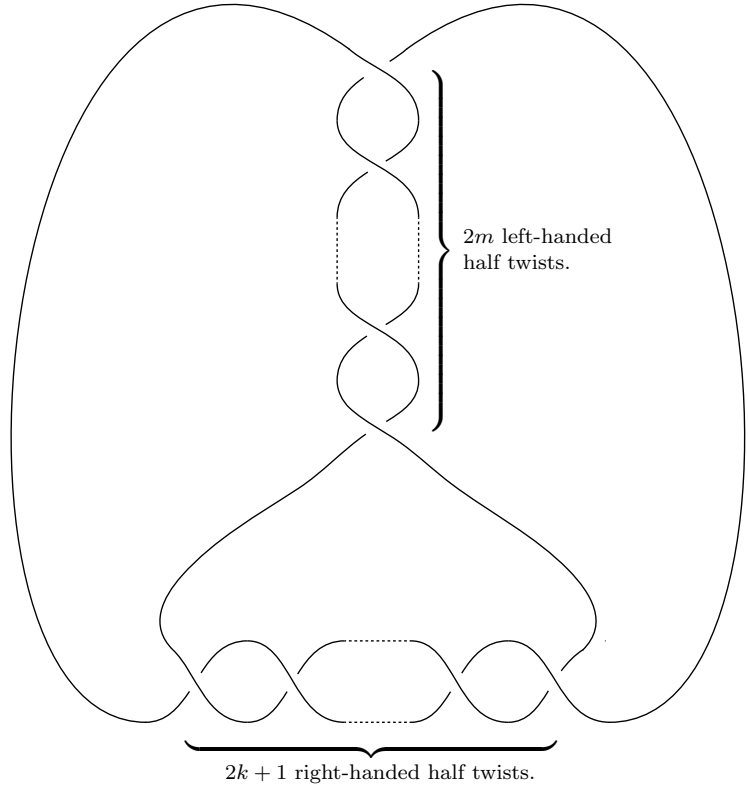
\end{example}
 
\subsection{Organization}
This article is organized into 4 sections. Section \ref{SectionBackgroundMaterial} discusses background material from Heegaard Floer homology. Section \ref{SectionProofs} is devoted to proving Theorem \ref{main} and Proposition \ref{PropositionAboutTau}. The final Section \ref{ExamplesSection} provides the missing calculations from the examples. 
\section{Background material} \label{SectionBackgroundMaterial}
\subsection{Homology of mapping cones} \label{hmc}
This section gathers a few facts about the homology of the mapping cone of a chain map between 
two chain complexes.
Let  
\begin{align} \nonumber
\mathcal{C} = \{ \ldots\to C_{i+1} \stackrel{\partial _{i+1}}{\to} C_i \stackrel{\partial _i}{\to} C_{i-1} \to \ldots\}, \cr
\mathcal{C}' = \{ \ldots\to C'_{i+1} \stackrel{\partial ' _{i+1}}{\to} C'_i \stackrel{\partial '_i}{\to} C '_{i-1} 
\to \ldots\}, 
\end{align}
be two finitely supported chain complexes of free Abelian groups. Let $f :  \mathcal{C} \to \mathcal{C}'$ 
be a chain map and let $f_i$ denote the restriction of $f$ to $C_i$. 
\begin{definition}
The mapping cone of $f:\mathcal C \to \mathcal C '$ is the complex 
$$\mathcal M =  \{ \ldots\to M_{i+1} \stackrel{D _{i+1}}{\longrightarrow} M_i \stackrel{D _i}{\longrightarrow} M_{i-1} \to \ldots\}
\quad \mbox{ with } \quad M_i = C_i \oplus C_{i+1} ' ,$$ 
and with $D=\{D_i\}_{i\in\mathbb N}$ defined as 
$$D_i(c,c') = \left( \partial _i c \, , \, \partial '_{i+1} c' + (-1)^i  f_i(c)\right),\quad (c,c') \in C_i \oplus C'_{i+1}.$$
\end{definition}
It is easy to verify that the maps $\iota:\mathcal C' \to \mathcal M$ and $\pi:\mathcal M \to \mathcal C$ 
defined by $ \iota(c') = (0,c')$ and $\pi(c,c') = c$ are chain maps that fit into the short exact sequence
$$0 \to C'_{i+1} \stackrel{\iota}{\to} M_i \stackrel{\pi}{\to} C_i \to 0. $$ 
The connecting homomorphism $\delta_i : C_i \to C_i'$ of this short exact sequence is given by $\delta _i = f_i$.
This discussion implies the next, easy but useful, theorem: 
\begin{theorem} \label{MappingConeTheorem}
Let $\mathcal M$ be the mapping cone of $f:\mathcal C \to \mathcal C'$. Then there is a long exact sequence
$$... \to H_{i+1}(\mathcal C)
 \stackrel{(f_{i+1})_*}{\to}  H_{i+1}(C')  \stackrel{\iota_*}{\to}  H_i(\mathcal M)  \stackrel{\pi_*}{\to}  H_i(\mathcal C)
 \stackrel{(f_i)_*}{\to}  H_i(\mathcal C') \to ...   $$
relating the homologies of $\mathcal C, \mathcal C'$ and $\mathcal M$. In particular, if $(f_{i+1})_* :H_{i+1} (\mathcal C)\to H_{i+1}(\mathcal C')$ is surjective then there is an isomorphism $ H_i(\mathcal M) \cong \kerr\, (f_i)_*$.
\end{theorem} 
\subsection{The Heegaard Floer groups} \label{HeegaardFloerGroupsDoneRight}
In \cite{OzsvathSzabo1, OzsvathSzabo2} P. Ozsv\'ath and Z. Szab\'o introduced chain complexes $CF^\infty (Y,\s)$, $CF^\pm (Y,\s)$ and $\widehat{CF}(Y,\s)$ associated to the choice of a pointed Heegaard diagram $\left( \Sigma _g, \{\alpha _1,...,\alpha _g\},\{\beta_1,...,\beta_g\},z\right)$\footnote{Here $z\in \Sigma_g$ is a point chosen in the complement of the $\alpha$ and $\beta$ attaching curves.}  for the closed and oriented $3$-manifold $Y$, and a choice of spin$^c$-structure $\s \in Spin^c(Y)$. Their homology groups $HF^\infty (Y,\s)$, $HF^\pm (Y,\s)$ and $\widehat{HF}(Y,\s)$ are topological invariants of $(Y,\s)$ and are referred to as the {\em Heegaard Floer homology groups of $(Y,\s)$}. 

The complex $CF^\infty (Y,\s)$ is freely generated by pairs $[x,i]$ with $i\in \mathbb Z$ and $x$ chosen from a finite set $ \mathcal X$ determined by the Heegaard diagram, subject to the relation $\s_z(x)=\s$, with $\s_z: \mathcal X\to Spin^c(Y)$ a function described in Section 2.6 of \cite{OzsvathSzabo1}. The complex comes equipped with an action of the polynomial ring $\mathbb Z[U]$ defined on generators by $U\cdot [x,i] = [x,i-1]$. The differential $\partial ^\infty$ of this complex has the property that $\partial ^\infty [x,i]$ is a sum of terms $[y,j]$ with $j\le i$. Accordingly, the subgroup  $CF^-(Y,\s)$ of $CF^\infty(Y,\s)$ generated by those $[x,i]$ with $i<0$ is a subcomplex. Their quotient complex is  $CF^+(Y,\s)$, while $\widehat{CF}(Y,\s)$ is the kernel of the chain map $U:CF^+(Y,\s)\to CF^+(Y,\s)$. Alternatively, if we view $CF^+(Y,\s)$ as a $\mathbb Z$-filtered chain complex with filtration $\mathcal F _+([x,i]) = i$, then $\widehat{CF}(Y,\s)=\mathcal F_+^{-1}(0)$. 

When $c_1(\s)$ is a torsion element of $H^2(Y;\mathbb Z)$, the associated Heegaard Floer homology groups carry a $\mathbb Q$-grading and we write $HF^\circ _{(d)}(Y,\s)$ to distinguish the summand of $HF^\circ (Y,\s)$ in grading $d\in \mathbb Q$, with $\circ \in \{ \infty, \pm, \widehat{\phantom{HF}}\}$. We shall also write $HF^\circ (Y,\s) = A_{(d_1)}\oplus B_{(d_2)}\oplus ...$ to express the same meaning, where $A,B,...$ are Abelian groups, for instance $\widehat{HF}(S^3)\cong \mathbb Z_{(0)}$.
\subsection{The knot Floer homology groups} \label{SectionKnotFloerHomologyGroups}
Ozsv\'ath and Szab\'o in \cite{OzsvathSzabo7} and J. Rasmussen \cite{Rasmussen1} introduced chain complexes $CFK^\infty (Y,K,\t)$ and $\widehat{CFK}(Y,K,\t,j)$ associated to a doubly pointed Heegaard diagram $\left( \Sigma _g, \{\alpha _1,...,\alpha _g\},\{\beta_1,...,\beta_g\},z,w\right)$\footnote{Here $z,w \in \Sigma_g$ are points in the complement of the $\alpha$ and $\beta$ attaching curves, chosen with respect to the knot $K\subset Y$.} for the pair $(Y,K)$ consisting of a closed and oriented $3$-manifold $Y$ and a null homologous knot $K\subset Y$, along with the choice of a spin$^c$-structure $\t \in Spin^c(Y_0(K))$ (here $Y_0(K)$ denotes the manifold obtained by zero surgery on $K$) and an integer $j\in \mathbb Z$.  Their homology groups $HFK^\infty (Y,K,\t)$ and $\widehat{HFK}(Y,K,\t,j)$ are the {\em knot Floer homology groups of $(Y,K,\t)$}.

The complex $CFK^\infty (Y,K,\t)$ is freely generated by triples $[x,i,j]$ with $i,j\in \mathbb Z$ and with $x\in \mathcal X$ subject to the relation $\underline{\s}(x)+(i-j)PD[\mu] = \t$, where $\underline{\s}:\mathcal X\to Spin^c(Y_0(K))$ is a function defined in Section 2.3 of \cite{OzsvathSzabo7}, and $PD[\mu]\in H^2(Y_0(K);\mathbb Z)$ is the Poincar\'e dual of the meridian $\mu$ of $K$. This complex too has an action of $\mathbb Z[U]$ given on generators by $U\cdot [x,i,j]=[x,i-1,j-1]$, and its differential $\partial ^\infty$ also has the property that $\partial ^\infty [x,i,j]$ is a sum of terms $[y,k,\ell]$ with $k\le i$ and $\ell \le j$. Thus, the  subgroups $C\{i\le k \}$ and $C\{i\le k, \, j\le \ell\}$ generated by those $[x,i,j]$ with $i\le k$, and $i\le k$ and $j\le \ell$ respectively, are subcomplexes of $CFK^\infty(Y,K,\t)$.

The function $\mathcal F_K:CFK^\infty (Y,K,\t) \to \mathbb Z^2$ defined by $\mathcal F_K([x,i,j]) = (i,j)$ renders $CFK^\infty (Y,K,\t)$ a $\mathbb Z^2$-filtered complex. This filtration induces a  $\mathbb Z$-filtration $\mathcal F_2= \Pi_2\circ \mathcal F_K$ (with $\Pi_i:\mathbb Z^2\to \mathbb Z$ being projection onto the $i$-th summand) on the quotient complex
$$C\{i=0\} :=\frac{C\{i\le0\}}{C\{i\le -1\}}.$$
The associated graded object of this filtered chain complex is $\widehat{CFK}(Y,K,\t,m)$, that is 
$$\widehat{CFK}(Y,K,\t,m) = \frac{\mathcal F_2^{-1}(\langle -\infty, m])}{\mathcal F_2^{-1}(\langle -\infty, m-1])}.$$
The generators of $\widehat{CFK}(Y,K,\t,m)$ are those $[x,0,m]$ with $\underline{\s}(x)=\t+mPD[\mu]$. 

Note that there is an isomorphism $Spin^c(Y_0(K))\cong Spin^c(Y)\oplus \mathbb Z$ of affine spaces, which sends a spin$^c$-structure $\t\in Spin^c(Y_0(K))$ to a pair $(\s,n)$ with $\s\in Spin^c(Y)$ obtained by the unique extension of $\t\big|_{Y_0(K) - K}$ to $Y$, and with $n=\frac{1}{2} \langle c_1(\t),[\hat F]\rangle$ where $F\subset Y$ is a Seifert surface of $K$ and $\hat F\subset Y_0(K)$ is obtained from $F$ by capping it off with the meridional disk of the knot $\hat K$ which is the core of the solid torus filling. Under this isomorphism, the $Spin^c(Y)$ component of $\underline{\s}(x)$ is $\s_z(x)$.

The knot Floer homology chain complex $CFK^\infty(Y,K,\t)$ comes equipped with a {\em \lq\lq conjugation map\rq\rq}, that is an isomorphism  $ J : CFK^\infty (Y,K,\t) \to CFK^\infty(Y,K,\t)$ which commutes with $\partial ^\infty$ and the action of $\mathbb Z[U]$. Formally, $J$ is induced by a reversal of the string orientation on $K$, but we shall not need this. The isomorphism $J$ induces an isomorphism (still denote by $J$)
\begin{equation} \label{TheJIsomorphism} 
J:C\{j= 0\}  \to C\{i=0\}.
\end{equation}
For this reason, one can compute $\widehat{HFK}(Y,K,\t,m)$ from $C\{j=0\}$, viewed as a filtered complex (with filtration $[x,i,0]\mapsto i$).

In our computations $J$ shall only play a secondary role, indeed, we shall only need to use the fact that $J$ is an isomorphism. 

For later use, we define a sequence of special chain complexes extracted from $CFK^\infty (Y,K,\t)$. Let $k,\ell$ be two integers and let $C\{i\le k, j\le \ell\}$ be the subcomplex of $CFK^\infty(Y,K,\t)$ generated by those $[x,i,j]$ with $i\le k$ and $j\le \ell$. For $s\in \mathbb Z$, define the chain complexes $\hat A_s$ and $\hat B$ as 
\begin{equation} \label{DefinitionOfAHat}
\hat A_s = \frac{C\{i\le 0, j\le s\}}{C\{i\le -1,j\le s-1\}} \quad \text{ and } \quad \hat B = C\{i=0\}.
\end{equation}
These complexes come with accompanying chain maps $\hat v_s, \hat h_s:\hat A_s\to \hat B$ defined as 
\begin{equation} \label{DefinitionOfVHatAndHHat}
\hat v_s([x,i,j]) = \left\{
\begin{array}{cl}
[x,0,j] & ; i=0,\cr
0 & ; i\ne 0,
\end{array}
\right.
\quad \quad 
\hat h_s([x,i,j]) = \left\{
\begin{array}{cl}
J([x,i-j,0]) & ; j=s,\cr
0 & ; j\ne s.
\end{array}
\right.
\end{equation}
Thus $\hat v_s$ is simply the projection map from $\hat A_s$ onto $\hat B$, cutting of the portion of $\hat A_s$ generated by those $[x,i,j]$ with $i<0$. Similarly, $\hat h_s$ is given by the action of $U^{-s}$, followed by projection onto $C\{j=0\}$, followed by $J$. We remark that $\hat A_s\cong \hat B$ whenever $s\ge g(K)$ in which case $\hat v_s$ is an isomorphism. In particular for all $s\ge g(K)$, $H_*(\hat A_s) \cong \widehat{HF}(S^3)\cong\mathbb Z$. Using the conjugation isomorphism $J$, one finds similarly that $H_*(\hat A_s)\cong \mathbb Z$ for all $s\le -g(K)$ and that $\hat h_s$ is an isomorphism in this range. We shall rely on this facts tacitly going forward. 

As was the case with Heegaard Floer groups, the knot Floer groups too carry a rational grading, provided $c_1(\s)$ is torsion (with $\t = (\s,m)$) and we shall similarly write, for example, $\widehat{HFK}_{(d)}(Y,K,\t,m)$ to single out the grading $d$ term of $\widehat{HFK}(Y,K,\t,m)$. Or we shall write $\widehat{HFK}(Y,K,\t,m)\cong A_{(d_1)}\oplus B_{(d_2)}\oplus ...$ for the same thing.

Knot Floer homology of $(Y,K)$ can be thought of as a $\mathbb Z$-filtration on Heegaard Floer homology of $Y$. Namely, the projection $\Pi:CFK^\infty (Y,K,\t) \to CF^\infty (Y,\s)$  (with $\t=(\s,m)$ under the above isomorphism $Spin^c(Y_0(K))\cong Spin^c(Y)\oplus \mathbb Z$) given on generators $\Pi([x,i,j]) = [x,i]$, is an isomorphism of chain complexes, and the composition $\Pi_1\circ\mathcal F_K \circ\Pi^{-1}:CFK^\infty (Y,\s) \to \mathbb Z$ is a filtration on $CF^\infty (Y,\s)$. The same map renders $\widehat{CF}(Y,\s)$, $\hat A_s$ and $\hat B$ into $\mathbb Z$-filtered complexes.  Applying the Leray spectral sequence to these filtered chain complexes, we find that 
\begin{itemize} 
\item[(i)] There is a Leray spectral sequence whose $E^2$-term is isomorphic, as a $\mathbb Z[U]$-module, to $\widehat{HFK}(Y,K,\t)\otimes _\mathbb Z \mathbb Z[U,U^{-1}]$, and that converges to $HF^\infty (Y,\s)$, and respects the rational gradings when $c_1(\s)$ is torsion. By $\widehat{HFK}(Y,K,\t)$ we mean $\oplus _{m\in \mathbb Z}\, \widehat{HFK}(Y,K,\t,m)$.
\item[(ii)] There is a Leray spectral sequence whose $E^2$-term is isomorphic to $\widehat{HFK}(Y,K,\t)$ and that converges to $\widehat{HF} (Y,\s)$, and that respects the rational gradings when $c_1(\s)$ is torsion. 
\item[(ii')] There is a Leray spectral sequence whose $E^2$-term is isomorphic to $\oplus_{j\in \mathbb Z} \widehat{HFK}(Y,K,\t,j)\otimes U^j$ and converges to $\widehat{HF} (Y,\s)$, and that respects the rational gradings when $c_1(\s)$ is torsion. This sequence is isomorphic to that from (ii) by using the isomorphism $J$ from \eqref{TheJIsomorphism}. 
\item[(iii)] There is a Leray spectral sequence whose $E^2$-term is isomorphic to 
$$\left(\bigoplus _{j\le s} \widehat{HFK}(Y,K,\t,j)\right) \oplus \left( \bigoplus _{j>s}  \widehat{HFK}(Y,K,\t,j)\otimes U^{j-s} \right)$$
and that converges to $H_*(\hat A_s)$, and that respects the rational gradings when $c_1(\s)$ is torsion. 
\end{itemize}
These spectral sequences are powerful computational tools that we shall heavily rely on.

In the case of $Y=S^3$ we shall simplify notation and write $CFK^\infty(K)$ for $CFK^\infty (S^3,K,\t_0)$ with $\t_0\in Spin^c(S^3_0((K)))$ characterized by $c_1(\t_0)=0$. We shall also write $\widehat{CFK}(K,j)$ for $\widehat{CFK}(S^3,K,\t_0,j)$, and we use similar notation for the homologies of these two chain complexes. 

\subsection{The rational surgery formula} 
This section describes the algorithm from \cite{OzsvathSzabo21} for the computation of $\widehat{HF}(S^3_{p/q}(K),\s)$ for a knot $K$ in $S^3$.

 For $i\in \mathbb{Z}$ let 
$$ \hat{\mathbb{A}}_i=\oplus _{s\in \mathbb{Z}} ( s,\hat A_{\lfloor \frac{i+ps}{q} \rfloor} ) \quad \quad 
 \text{ and } \quad \quad \hat{\mathbb{B}}= \hat{\mathbb{B}}_i=\oplus _{s\in \mathbb{Z}} (s,\hat B).\footnote{The index $i$ in $\hat{\mathbb{B}}_i$ if for bookkeeping purposes only, the complex $\hat{\mathbb{B}}_i$ is independent of $i$. When convenient we will write $\hat{\mathbb{B}}$ instead of $\hat{\mathbb{B}}_i$. }  $$ 
In the above, both $(k,\hat A_\ell )$ and $(k,\hat B)$ denote copies of $\hat A_\ell$ and $\hat B$ respectively and $\lfloor x \rfloor$ is the largest integer smaller than or equal to $x$. We use the maps $\hat v_k, \hat h_k : \hat A_k \to \hat B$  to define maps $\hat v, \hat h:\hat{\mathbb{A}}_i \to \hat{\mathbb{B}}_i$ by the convenction that $\hat v$ maps  $( s,\hat A_{\lfloor \frac{i+ps}{q} \rfloor})$ to $(s, \hat B)$ via $\hat v_{\lfloor \frac{i+ps}{q} \rfloor}$, while $\hat h$ maps $( s,\hat A_{\lfloor \frac{i+ps}{q} \rfloor} )$ to $(s-1, \hat B)$ via 
$\hat h_{\lfloor \frac{i+ps}{q} \rfloor}$.
Setting $r=p/q$, we define the chain map $\hat D_{i,r}:\hat{\mathbb{A}}_i\to \hat{\mathbb{B}}_i$ as 
$$ \hat D_{i,r} \left( \{(s,a_s)\}_{s\in \mathbb{Z}} \right) = \{(s,b_s)\}_{s\in \mathbb{Z}} \quad \quad 
\mbox{ with } \quad \quad 
b_s = \hat v_{\lfloor \frac{i+ps}{q} \rfloor}(a_s) + \hat h_{\lfloor \frac{i+p(s-1)}{q} \rfloor} (a_{s-1}). $$
Let $\hat{\mathbb{X}}_{i,r}$ be the mapping cone of $\hat D_{i,r}$. Note that that $\mathbb{X}_{i,r} = \mathbb{X}_{i',r}$ whenever $i$ and $i'$ are congruent modulo $p$. Given this we further modify our notation to $\hat{\mathbb{X}}_{[i],r}$ where $[i]\in \mathbb Z/p\mathbb Z$ is the equivalence class of $i\in \mathbb Z$ modulo $p$. 
\begin{theorem}[Ozsv\'ath-Szab\'o \cite{OzsvathSzabo21}] \label{RationalSurgeryTheorem}
Let $K\subset S^3$ be a knot and let $p,q\in\mathbb{Z}$ be a pair of relatively prime, nonzero integers. Then there is an affine identification of $Spin^c(S^3_{p/q}(K))$ with $\mathbb Z/p\mathbb Z$, with respect to which there is an isomorphism  
$$ \widehat{HF}(S^3_{p/q}(K),\s) \cong H_*(\hat{\mathbb{X}}_{[i],r}).$$
\end{theorem}
Theorem \ref{RationalSurgeryTheorem} is the main instrument for the proofs of our results, and we pause before proceeding to give a simple example. The key ingredient to using Theorem \ref{RationalSurgeryTheorem} is an understanding of the groups $H_*(\hat A_s)$ (the homology $H_*(\hat B)$ is isomorphic to $\mathbb Z_{(0)}$ for knots in $S^3$) and the maps $\hat v_s, \hat h_s$. We shall get a handle on both by using of the Leray spectral sequences from Section \ref{SectionKnotFloerHomologyGroups}.  

To begin with, we introduce a  useful way of conceptualizing the chain complex $CFK^\infty (K)$. Namely, we represent each of the groups $\widehat{HFK}(K,j)$ by a dot in a coordinate plane, placed at the point $(0,j)$. We let $U$ act by translation by the vector $(-1,-1)$, and thus fill out a diagonal region of the plane with dots representing the groups $\widehat{HFK}(K,j)\otimes U^m$ (the latter sitting at coordinates $(-m,-m+j)$). We shall refer to the horizontal coordinate as the $i$-coordinate and the vertical one as the $j$-coordinate. Thus, the groups in the entire $ij$-plane are the $E^2$ term of the Leray spectral sequence converging to $HF^\infty (S^3)$, the $j$-axis contains the $E^2$ term converging to $\widehat{HF}(S^3)$ and the \lq\lq angle\rq\rq of points $(i,j)$ with $\min (i,j-s)=0$ represents the $E^2$ term converging to $H_*(\hat A_s)$. Knowing the $E^\infty$ terms of the first two sequences typically lets one pin down the differentials of the higher order terms of the spectral sequence, and use the third of these sequences to compute $H_*(\hat A_s)$. Here is an example. 

\begin{example} Consider the $(3,4)$ torus knot $K=T_{(3,4)}$. It's knot Floer homology can be computed from the results of \cite{OzsvathSzabo12}:
$$\widehat{HFK}(T_{(3,4)},j) \cong \left\{
\begin{array}{cl}
\mathbb Z_{(0)} & \quad ; \quad j=3, \cr
\mathbb Z_{(-1)} & \quad ; \quad j=2, \cr
\mathbb Z_{(-2)} & \quad ; \quad j=0, \cr
\mathbb Z_{(-5)} & \quad ; \quad j=-2, \cr
\mathbb Z_{(-6)} & \quad ; \quad j=-3. \cr
\end{array}
\right. 
$$
The spectral sequence (ii) from Section \ref{SectionKnotFloerHomologyGroups} converging to $\widehat{HF}(S^3)\cong \mathbb Z_{(0)}$ shows that there are two \lq\lq vertical\rq\rq \, higher differentials on its $E^2$ term, namely $d_2:\mathbb Z_{(-1)} \to \mathbb Z_{(-2)}$ and $d_2:\mathbb Z_{(-5)} \to \mathbb Z_{(-6)}$, both isomorphisms, and the spectral sequence abuds after this level. From the spectral sequence (ii'), we find similar \lq\lq horizontal\rq\rq\, differentials. Finally, applying sequence (i), shows that there are no further \lq\lq diagonal\rq\rq differentials in the $ij$-plane, see Figure \ref{pic4}.
\begin{figure}[htb!] 
\centering
\includegraphics[width=12cm]{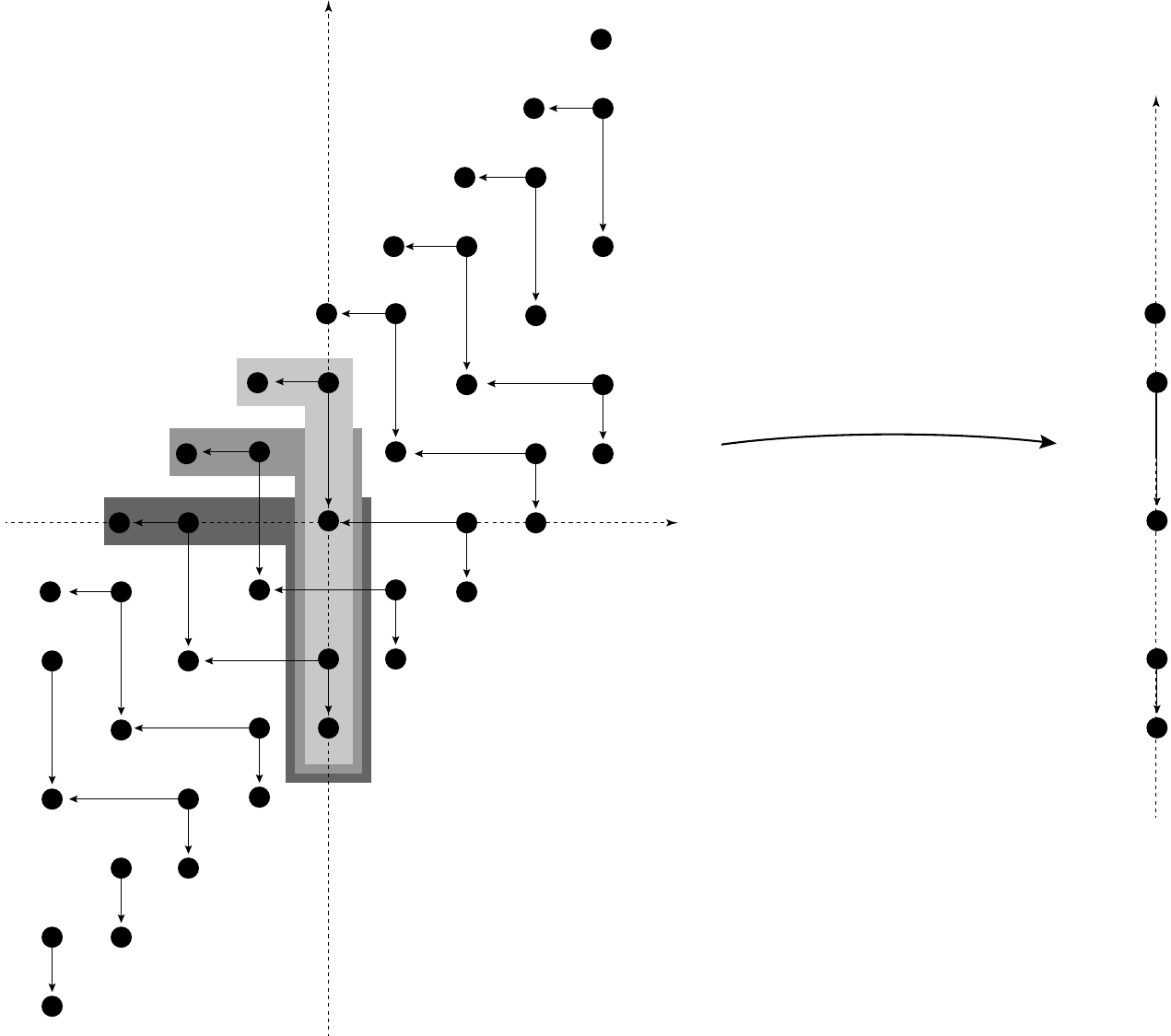}
\put(-100,180){$\hat v_s, \hat h_s$}
\caption{Example of the $(3,4)$-torus knot. The homology of the shaded regions give $H_*(\hat A_s) \cong \mathbb Z$ for $s=2$ (lightest shading), $s=1$ (medium shading) and $s=0$ (darkest shading). In each case $\hat v_s=0$ and $\hat h_s=0$.}  \label{pic4}
\end{figure}

Having used spectral sequences (i), (ii) and (ii') from Section \ref{SectionKnotFloerHomologyGroups} to pin down the differentials in the sequence (i), we now turn to the spectral sequence (iii) to compute the homologies $H_*(\hat A_s)$. Using Figure \ref{pic4}, it is now easy to find that $H_*(\hat A_s)\cong \mathbb Z$ for all $s\in \mathbb Z$ and that 
$$\hat v_s =  \left\{
\begin{array}{cl}
\text{id} & \quad ; \quad s\ge 3, \cr
0 & \quad ; \quad s\le 2,
\end{array}
\right. 
\quad \text{ and } \quad 
\hat h_s =  \left\{
\begin{array}{cl}
0 & \quad ; \quad s\ge -2, \cr
\text{id} & \quad ; \quad s\le -3.
\end{array}
\right. 
$$
With these in place, it is now easy to use Theorem \ref{RationalSurgeryTheorem}. We invite the reader to check that for instance $\widehat {HF}(S^3_{-1}(T_{(3,4)}))\cong \mathbb Z^{11}$. 
\end{example}
%
\subsection{The Ozsv\'ath-Szab\'o $\tau$-invariant} \label{SectionTheTauInvariant}
In \cite{OzsvathSzabo11} Ozsv\'ath and Szab\'o introduced a concordance invariant $\tau(K)$ for a knot $K$ in $S^3$. It is defined as 
$$\tau(K) = \min \{j\in \mathbb Z  \, |\, (\iota_j)_* :H_*(\mathcal F_K^{-1}(-\infty,j])) \to \widehat{HF}(S^3) \text{ is nontrivial.} \},$$
where $\iota_j:\mathcal F_K^{-1}(-\infty,j]) \to \widehat{CF}(S^3)$ is the inclusion map. The definition is well posed as $\iota_j$ is an isomorphism for all sufficiently large $j$. 
\section{Proofs} \label{SectionProofs}
\subsection{Proof of Theorem \ref{main}}
As in the setup of Theorem \ref{main}, let $K\subset S^3$ be a knot of genus $g\ge 1$, and let $p,q$ be two nonzero and relatively prime integers with $p>0$. Let $Y=S^3_{p/q}(K)$ be the results of $p/q$-framed Dehn surgery on $K$ and let $\ell$ be the number of $L$-structures on $Y$. Recalling the identification $Spin^c(S^3_{p/q}(K))$ with $\mathbb Z/p\mathbb Z$ from Theorem \ref{RationalSurgeryTheorem}, we shall label spin$^c$-structures on $S^3_{p/q}(K)$ by $[i]$, the equivalence class in $\mathbb Z/p\mathbb Z$ of the integer $i$. Our goal then is to demonstrate the validity of the inequality 
\begin{equation} \label{maininequality2}
2g-1 \ge \frac{p-\ell}{|q|}.
\end{equation}

We note firstly that is suffices to establish \eqref{maininequality2} for  $q>0$. For if $q<0$ and $Y=S^3_{p/q}(K)$, then $-Y = S^3_{-p/q}(\bar K)$ where $\bar K$ is the mirror image of $K$. The genus of $\bar K$ equals that of $K$, and the number of $L$-structures on $-Y$ equals that on $Y$ \cite{OzsvathSzabo2}.
Thus, inequality \eqref{maininequality2} for $p/q$-surgery on $K$ is established by establishing it for $-p/q$-surgery on $\bar K$.  

Assuming $p,q>0$, we prove Theorem \ref{main} following a twofold strategy:
\begin{itemize}
\item[(a)] We shall first count the number of spin$^c$-structures $[i]\in Spin^c(S^3_{p/q}(K))$ for which either $\lfloor \frac{i+ps}{q}\rfloor \ge g$ or $\lfloor \frac{i+ps}{q}\rfloor \le -g$ for all values $s\in \mathbb Z$, and show that there is $\max (p-(2g-1)q,0)$ such spin$^c$-structures. 
\item[(b)] We shall show that each spin$^c$-structure $[i]\in Spin^c(S^3_{p/q}(K))$ from part (a) is an $L$-structure. 
\end{itemize}
Put together, these two claims show that the number $\ell$ of $L$-structures on $Y$ satisfies the inequality
$$\ell \ge \left\{ 
\begin{array}{cl}
p-(2g-1)q & \quad ; \quad p-(2g-1)q \ge 0, \cr
0 & \quad ; \quad p-(2g-1)q\le 0. 
\end{array}
\right. 
$$
If $p-(2g-1)q\ge 0$ then the inequality $\ell \ge p-(2g-1)q$ readily transforms into inequality \eqref{maininequality2}. While the inequality $\ell \ge 0$ is without content, it occurs when $p-(2g-1)q\le 0$ giving $2g-1\ge p/q$ and clearly $p/q\ge (p-\ell)/q$, establishing \eqref{maininequality2} once more.  With this understood, we turn to proving Claims (a) and (b). 

\begin{lemma} \label{LemmaAux1}
Let $p,q$ be positive integers. Then there exist spin$^c$-structures $[i]\in Spin^c(S^3_{p/q}(K))$ for which either $\lfloor \frac{i+ps}{q}\rfloor \ge g$ or $\lfloor \frac{i+ps}{q}\rfloor \le -g$ for all values $s\in \mathbb Z$, precisely when $p> (2g-1)q$. If this inequality is met, then the said spin$^c$-structures $[i]$ are the equivalence classes of 
the set $\{gq,...,p+q-gq-1\}$. In particular there are $p-(2g-1)q$ such spin$^c$-structures.  
\end{lemma}
\begin{proof}
Let $[i]\in \mathbb Z/p\mathbb Z$ be a spin$^c$-structure, and for simplicity let us agree to choose $i$ from the set $\{0,...,p-1\}$. Since $p$ and $q$ have been assumed to be positive, then the function $s\mapsto \lfloor \frac{i+ps}{q}\rfloor$ is non-decreasing. Additionally, note that $\lfloor \frac{i+ps}{q}\rfloor<0$ for $s<0$ and $\lfloor \frac{i+ps}{q}\rfloor\ge 0$ for $s\ge 0$. Accordingly, 
$$ \lfloor \textstyle \frac{i+ps}{q}\rfloor\le \lfloor \frac{i-p}{q}\rfloor, \text{ for all } s\le -1 \quad \quad \text{ and } \quad \quad  \lfloor \textstyle \frac{i+ps}{q}\rfloor\ge \lfloor \frac{i}{q}\rfloor, \text{ for all } s\ge 0.$$ 
Therefore, if $i\in \{0,...,p-1\}$ is such that $ \lfloor \textstyle \frac{i-p}{q}\rfloor\le-g$ and $ \lfloor \textstyle \frac{i}{q}\rfloor\ge g$, it will satisfy the requirement of the lemma. The set of $i\in \{0,...,p-1\}$ that obey the inequality $ \lfloor \textstyle \frac{i}{q}\rfloor\ge g$ is $\{gq,...,p-1\}$ if $gq\le p-1$ (and is otherwise empty), while the set of those $i$ for which $ \lfloor \textstyle \frac{i-p}{q}\rfloor\le-g$, is given by $\{0,...,p+q-gq-1\}$ if $p+q-gp-1\ge 0$ (and is otherwise empty). The intersection of these two sets is $\{gq,...,p+q-gq-1\}$ which has cardinaliy $p-(2g-1)q$, and is nonempty if and only if $p-1\ge gq$ and $p+q-1\ge gq$ and $p+q>2gq$. The third of these inequalities implies the first two, and the claim of the lemma follows.
\end{proof}
\begin{lemma} \label{LemmaAux2}
Let $K\subset S^3$ be a knot of genus $g\ge 1$ and let $p,q$ be relatively prime integers. Let $[i]\in Spin^c(S^3_{p/q}(K))$ be a spin$^c$-structure on $Y=S^3_{p/q}(K)$ for which either $\lfloor \frac{i+ps}{q}\rfloor \ge g$ or $\lfloor \frac{i+ps}{q}\rfloor \le -g$ for all values $s\in \mathbb Z$. Then $[i]$ is an $L$-structure. 
\end{lemma}
\begin{proof}
Assume firstly that $p,q$ are positive. 
For convenience let us choose $i$ from the set $\{0,...,p-1\}$ so that the assumption of the lemma leads to  $\lfloor \frac{i+ps}{q}\rfloor \ge g$ for all $s\ge 0$ and  $\lfloor \frac{i+ps}{q}\rfloor \le -g$ for all $s<0$. Accordingly, 
$$H_*\left(\hat A_{\lfloor \frac{i+ps}{q}\rfloor} \right) \cong \left\{
\begin{array}{cl}
H_*(C\{i =  0\}) \cong \mathbb Z & \quad ; \quad s\ge 0, \cr
& \cr 
H_*(C\left\{ j=\lfloor \frac{i+ps}{q}\rfloor \right\}) \cong \mathbb Z & \quad ; \quad s<0.
\end{array}
\right.
$$
Additionally, the maps in homology induced by  $\hat v_k$ and $\hat h_k$ (still denoted $\hat v_k$ and $\hat h_k$) are given by 
$$\hat v_{\lfloor \frac{i+ps}{q}\rfloor} = \left\{
\begin{array}{cl}
\text{id} & ; \quad s\ge 0, \cr
0 & ; \quad s<0,
\end{array}
\right. 
\quad \quad \text{ and } \quad \quad
\hat h_{\lfloor \frac{i+ps}{q}\rfloor} = \left\{
\begin{array}{cl}
0 & ; \quad s\ge 0, \cr
J\circ \Pi_{C\{j=0\}}\circ U^{\lfloor \frac{i+ps}{q}\rfloor} & ; \quad s<0,
\end{array}
\right. 
$$
with $\Pi_{C\{j=0\}}$ being projection onto $C\{j=0\}$. Though we may not know the map $J$ explicitly, we note that $\hat h_{\lfloor \frac{i+ps}{q}\rfloor}$ is an isomorphism for all $s<0$. 

Consider now the mapping cone $\hat {\mathbb X}_{[i],r}$ (with $r=p/q$) of $\hat v+\hat h : \hat{\mathbb A}_i \to \hat{\mathbb B}_i$. The explicit formulae for $\hat v$ and $\hat h$ above show that the map in homology induced by $\hat v+\hat h$ is onto, and so in light of Theorem \ref{MappingConeTheorem} we obtain the isomorphism 
$$\widehat{HF}(S^3_{p/q}(K),[i]) \cong H_*(\hat{\mathbb X}_{[i],r}) \cong \kerr \left(\hat v+\hat h:H_*(\hat{\mathbb A}_i) \to H_*(\hat{\mathbb B}_i)\right).$$
The kernel of $\hat v+\hat h$ is easily computed. Namely, consider the following diagram in which the vertical maps indicate the nonzero $\hat v_k$'s and the slanted maps correspond to the nonzero $\hat h_k$'s:

\centerline{\tiny
\xymatrix{
H_*\left((-2, \hat A_{\left\lfloor\frac{i-2p}{q} \right\rfloor})\right) \ar[dr]_{\cong} & 
H_*\left((-1, \hat A_{\left\lfloor\frac{i-p}{q} \right\rfloor})\right)  \ar[dr]_{\cong} & 
H_*\left((0,\hat A_{\left\lfloor\frac{i}{q} \right\rfloor})\right) \ar[d]^{\cong}& 
H_*\left((1,\hat A_{\left\lfloor\frac{i+p}{q} \right\rfloor})\right)  \ar[d]^{\cong}  \\
 & 
H_*((-1,\hat B)) & 
H_*((0,\hat B)) & 
H_*((1,\hat B))  \\ 
}
} 
\vskip1mm

The direct sum of the groups in the top row (which is infinite in both directions) represents $H_*(\hat{\mathbb{A}}_i)$ while the direct sum of the groups in the bottom row (likewise infinite in both directions) represents $H_*(\hat{\mathbb{B}}_i)$. The kernel of $\hat v+\hat h$ is easily explicitly identified as 
$$\kerr (\hat v+\hat h) = \left\{  \{(s,a_{\left\lfloor\frac{i+ps}{q} \right\rfloor})\}_{s\in \mathbb Z}\, \big| \,  a_{\left\lfloor\frac{i+ps}{q} \right\rfloor}=0  \text{ for $s\ne -1,0$ and } a_{\left\lfloor\frac{i-p}{q} \right\rfloor}+a_{\left\lfloor\frac{i}{q} \right\rfloor}=0\right\}.$$
Clearly $\kerr (\hat v+\hat h)\cong \mathbb Z$ as needed. 

In the case where $p/q<0$, the above argument needs slight modification. Specifically, the homology of $\hat{\mathbb X}_{[i],r}$ is computed as the homology of the mapping cone

\centerline{\tiny
\xymatrix{
& 
 H_*\left((-1, \hat A_{\left\lfloor\frac{i-p}{q} \right\rfloor})\right) \ar[dl]_{\cong}  & 
H_*\left((-1, \hat A_{\left\lfloor\frac{i-p}{q} \right\rfloor})\right)  \ar[dl]_{\cong} & 
H_*\left((0,\hat A_{\left\lfloor\frac{i}{q} \right\rfloor})\right) \ar[d]^{\cong}& 
H_*\left((1,\hat A_{\left\lfloor\frac{i+p}{q} \right\rfloor})\right)  \ar[d]^{\cong}  \\
H_*((-3,\hat B)) & 
H_*((-2,\hat B)) & 
H_*((-1,\hat B)) & 
H_*((0,\hat B)) & 
H_*((1,\hat B))  \\ 
}
} 
\vskip1mm

This time $\hat v+\hat h$ is into, rather than being onto, and an application of Theorem  \ref{MappingConeTheorem} shows that $H_*(\hat{\mathbb X}_{[i],r})\cong \text{Coker}(\hat v +\hat h) \cong H_*((-1,\hat B))\cong \mathbb Z$. 
\end{proof}
\subsection{Proof of Proposition \ref{PropositionAboutTau}}
Let $K$ be a knot of genus $g>1$ and assume firstly that $\tau(K) = g$. Let $p,q>0$ be relatively prime integers with  $p>(2g-1)q$. This inequality assures that all spin$^c$-structures $[i]$ satisfy one of two properties: 
\begin{itemize}
\item Either $\lfloor \frac{i+ps}{q}\rfloor \le -g $ or $\lfloor \frac{i+ps}{q}\rfloor \ge g$ for all $s\in \mathbb Z$. 
\item There exists exactly one $s_i\in \mathbb Z$ such that $-g<\lfloor \frac{i+ps_i}{q}\rfloor <g$. 
\end{itemize}
Spin$^c$-structures of the first kind are $L$-structures and there are exactly $p-(2g-1)q$ of them (Lemmas \ref{LemmaAux1} and  \ref{LemmaAux2}). Turning to spin$^c$-structures $[i]$ of the second kind, we note that the assumption of $\tau(K)  = g$ implies the vanishing of certain of the maps $\hat v_s$ and $\hat h_s$. Namely, consider the factorization $\hat v_s = \iota_s\circ \pi_s$  
$$\hat A_s \stackrel{\pi_s}{\longrightarrow} \mathcal F_K^{-1}(\langle -\infty, s]) \stackrel{\iota_s}{\longrightarrow} \hat B = \widehat {CF}(S^3)$$ 
with $\pi_s$ being the projection and $\iota_s$ the inclusion map. Since $\tau(K)=g$, it follow that $\hat v_s=0$ for all $s<g$ (since $\iota_s=0$ for $s<g$) and similarly that $\hat h_s =0$ for all $s>-g$. Of course, $\hat v_s$ is an isomorphism for all $s\ge g$ and $\hat h_s$ is an isomorphism for all $s\le -g$. Accordingly, the homology of $\hat{\mathbb X}_{[i],r}$ (with $r=-p/q$) is the homology of the mapping cone

\centerline{
 \xymatrix{
& 
\mathbb Z  \ar[dl] & 
\mathbb Z  \ar[dl] & 
H_*(\hat A_{s_i}) & 
\mathbb Z  \ar[d] &
\mathbb Z  \ar[d]   \\
\mathbb Z  & 
\mathbb Z  & 
\mathbb Z & 
\mathbb Z & 
\mathbb Z & 
\mathbb Z  \\ 
}
}
\vskip1mm
\noindent Theorem \ref{MappingConeTheorem} implies then that $\mathbb Z^2$ injects into $H_*(\hat{\mathbb X}_{[i],r})$ and so $[i]$ is not an $L$-structure. 

The case of $\tau(K) =-g$ and $r=p/q$ follows by symmetry since $\widehat{HF}(S^3_{p/q}(K)) \cong \widehat{HF}(S^3_{-p/q}(\bar K))$ and $\tau(\bar K)=-\tau(K)$. 
\section{Examples} \label{ExamplesSection}
In this section we provide computations supporting our claims in Examples \ref{ExampleDiscrepancy} -- \ref{ExampleOfNonLMinimizingKnots} from the introduction. The main tools are Theorems \ref{MappingConeTheorem} and \ref{RationalSurgeryTheorem} which are used to compute the Heegaard Floer homology of a rational surgery $S^3_{p/q}(K)$ on a knot $K$, with the homologies $H_*(\hat A_s)$ and the maps $\hat v_s, \hat h_s :  H_*(\hat A_s) \to H_*(\hat B)$ as input. The latter groups and maps are computed with the help of the spectral sequences (i)--(iii) from Section \ref{SectionKnotFloerHomologyGroups}. 

Example \ref{ExampleSharpness} follows directly from Proposition \ref{PropositionAboutTau} and the fact that $\tau(T_{(2,2g+1)})=g$.
\subsection{Computations for Example \ref{ExampleDiscrepancy}} For a positive integer $n$, let $Y_n$ be the result of $-\frac{4n+1}{n}$-framed surgery on the Figure Eight knot. The spin$^c$-structures on $Y_n$ can be grouped into two disjoint categories:
\begin{itemize}
\item Spin$^c$-structures $[i]\in \mathbb Z/(4n+1)\mathbb Z$ for which either $\lfloor \frac{i+(4n+1)s}{n}\rfloor \le -1$ or $\lfloor \frac{i+(4n+1)s}{n}\rfloor \ge 1$ for all $s\in \mathbb Z$. 
\item Spin$^c$-structures $[i]\in \mathbb Z/(4n+1)\mathbb Z$ for which there exists a unique $s_i\in \mathbb Z$ such that  $\lfloor \frac{i+(4n+1)s_i}{n}\rfloor =0$.
\end{itemize}
According to Lemma \ref{LemmaAux1} there are $3n+1$ spin$^c$-structures of the first kind, and each of them is an $L$-structure. To show that $\ell (Y_n) = 3n+1$, we need to demonstrate that none of the spin$^c$-structures of the second kind is in an $L$-structure. This becomes an explicit calculation, after determining the differentials in the spectral sequence (i) from Section \ref{SectionKnotFloerHomologyGroups}. Since the Figure Eight knot  $4_1$ is alternating, with Alexander polynomial $t-3+t^{-1}$, and with vanishing signature, its knot Floer homology is given by 
$$\widehat{HFK}(4_1,j) \cong \left\{
\begin{array}{ll}
\mathbb Z_{(1)} & \quad ; \quad j= 1,\cr
\mathbb Z^3_{(0)} & \quad ; \quad j= 0,\phantom{\int_a^b}\cr
\mathbb Z_{(-1)} & \quad ; \quad j= -1.
\end{array}
\right. 
$$
There is a pair of non vanishing vertical differentials $d_2:\mathbb Z_{(1)}\to \mathbb Z^3_{(0)}$ and $d_2:\mathbb Z^3_{(0)}\to \mathbb Z_{(-1)}$ given by the inclusion into the first coordinate, and projection onto the third coordinate, respectively.  Their horizontal counterparts looks as in Figure \ref{pic5} and there are no other differentials. 
\begin{figure}[htb!] 
\centering
\includegraphics[width=5cm]{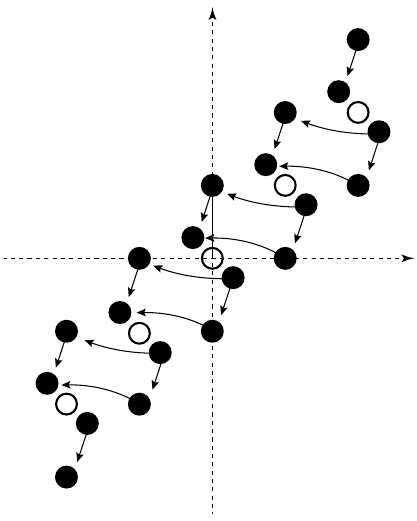}
\caption{The differentials in the spectral sequence $E^2=\widehat{HF}(4_1)\otimes _\mathbb Z \mathbb Z[U,U^{-1}]$ converging to $E^\infty = HF^\infty (S^3)$. The white dot represents the summand $\mathbb Z_{(0)}$ (and its various translates by $U^n$) containing the generator of $\widehat{HF}(S^3)$, the two blacks dots right next to it represent the other two copies of $\mathbb Z_{(0)}$ in $\widehat{HFK}(4_1,0)$ (and their various translates by $U^n$). }  \label{pic5}
\end{figure}

From this, it is now an easy matter to find that 
$$
H_*(\hat A_s) \cong \left\{
\begin{array}{cl}
\mathbb Z  & ; s\ne 0, \cr
\mathbb Z\oplus \mathbb Z^2 & ; s=0, 
\end{array}
\right.
\quad \quad 
\hat v_s = \left\{
\begin{array}{cl}
\text{id} & ; s>0,\cr
\pi_1 & ; s=0,\cr
0 & ; s<0,
\end{array}
\right.
\quad \quad 
\hat h_s = \left\{
\begin{array}{cl}
0 & ; s> 0,\cr
\pi_1& ; s=0, \cr
\text{id} & ; s<0.
\end{array}
\right.
$$
In the above, $\pi_1:\mathbb Z\oplus \mathbb Z^2\to \mathbb Z$ is the projection on the first summand. These computations show that $\widehat{HF}(Y,[i])\cong \mathbb Z^3$ for all spin$^c$-structures $[i]$ of the second kind,  as this Heegaard Floer group is the homology of the mapping cone below (with arrows without labels corresponding to  isomorphisms).

\centerline{\tiny
\xymatrix{
& 
\mathbb Z  \ar[dl] & 
\mathbb Z  \ar[dl] & 
\mathbb Z\oplus \mathbb Z^2 \ar[dl]_{\pi_1} \ar[d]^{\pi_1}& 
\mathbb Z  \ar[d] &
\mathbb Z  \ar[d]   \\
\mathbb Z  & 
\mathbb Z  & 
\mathbb Z & 
\mathbb Z & 
\mathbb Z & 
\mathbb Z  \\ 
}
}
\vskip1mm
\noindent It follows that $\ell (Y_n) = 3n+1$ as claimed, demonstrating the sharpness of inequalities \eqref{MainInequality} and \eqref{MainInequalityCorollary} from Theorem \ref{main} and Corollary \ref{main2} (in the case of $q=1$) respectively. In particular, $\gq(Y_n) =1$ and $\gz(Y_n)\ge \frac{n+1}{2}$ and hence $\gz(Y_n) - \gq(Y_n)\ge \frac{n-1}{2}$. 

To show that $\gz(Y_n)$ is finite, we show that $Y_n$ also arises as an integral surgery on a knot. We will do so by applying a set of Rolfsen twists \cite{Rolfsen2, Rolfsen1} to the framed knot $(4_1, -\frac{4n+1}{n})$ through which $Y_n$ was defined. 

Consider first the framed link from Figure \ref{pic1}(a). After applying a Rolfsen twist to its component with framing $1$ (and after discarding the resulting $\infty$-framed unknot) we arrive at the framed knot in Figure \ref{pic1}(b), yielding $(4_1,-\frac{4n+1}{n})$ after a simple isotopy. Thus $Y_n$ is also the result of Dehn surgery on the framed link in Figure \ref{pic1}(a). 
\begin{figure}[htb!] 
\centering
\includegraphics[width=15cm]{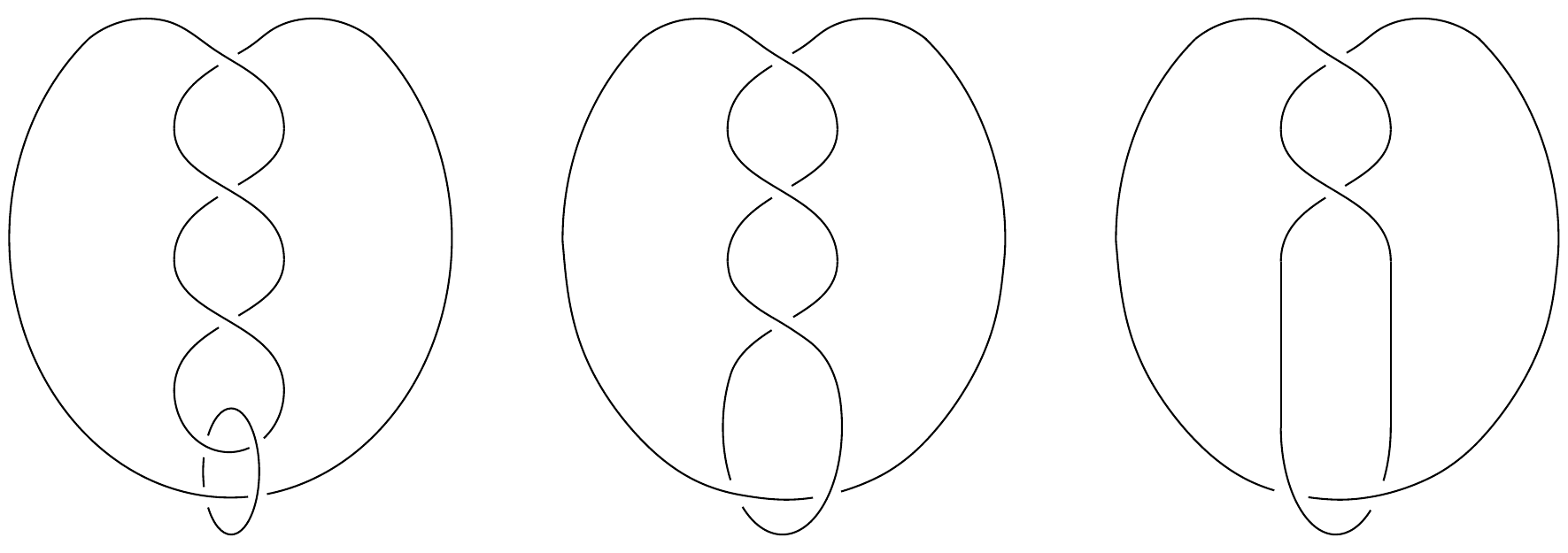}
\put(-360,-5){$1$}
\put(-420,20){$-\frac{1}{n}$}
\put(-285,20){$-\frac{4n+1}{n}$}
\put(-133,20){$-\frac{4n+1}{n}$}
\put(-370,-30){(a)}
\put(-220,-30){(b)}
\put(-70,-30){(c)}
\caption{A Rolfsen twist on the $1$-framed unknot in Figure (a) yields the framed knot in Figure (b). Figure (c) results from the latter by a simple isotopy. }  \label{pic1}
\end{figure}
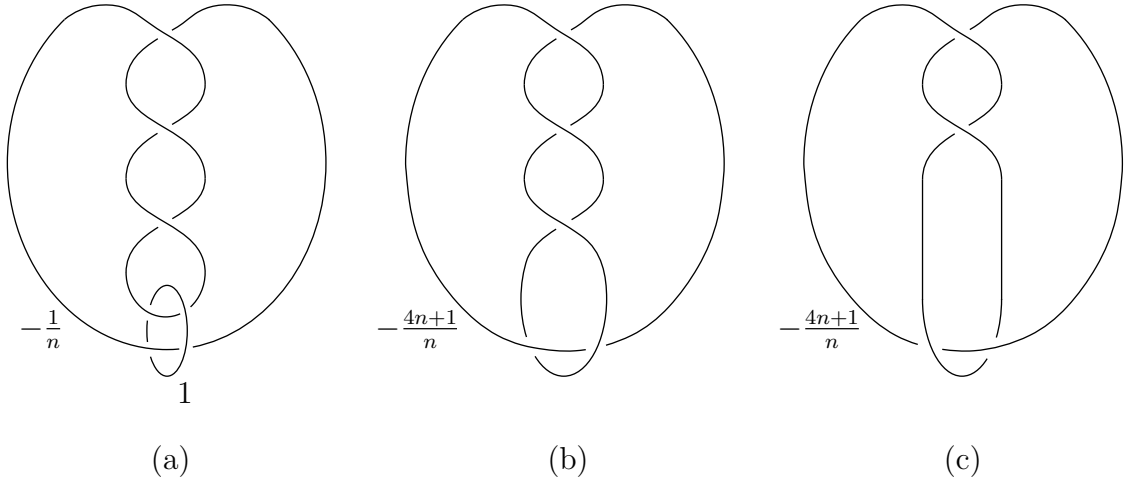

Applying an isotopy to the framed link from Figure \ref{pic1}(a) gives the framed link in Figure \ref{pic2}(a). The latter, after performing a Rolfsen twist on the  $-\frac{1}{n}$-framed component (and again discarding the resulting $\infty$-framed unknot) leads to the framed knot in Figure \ref{pic2}(b). The framing of the latter is an integer, showing that $\gz(Y_n)$ is finite. 
\begin{figure}[htb!] 
\centering
\includegraphics[width=14cm]{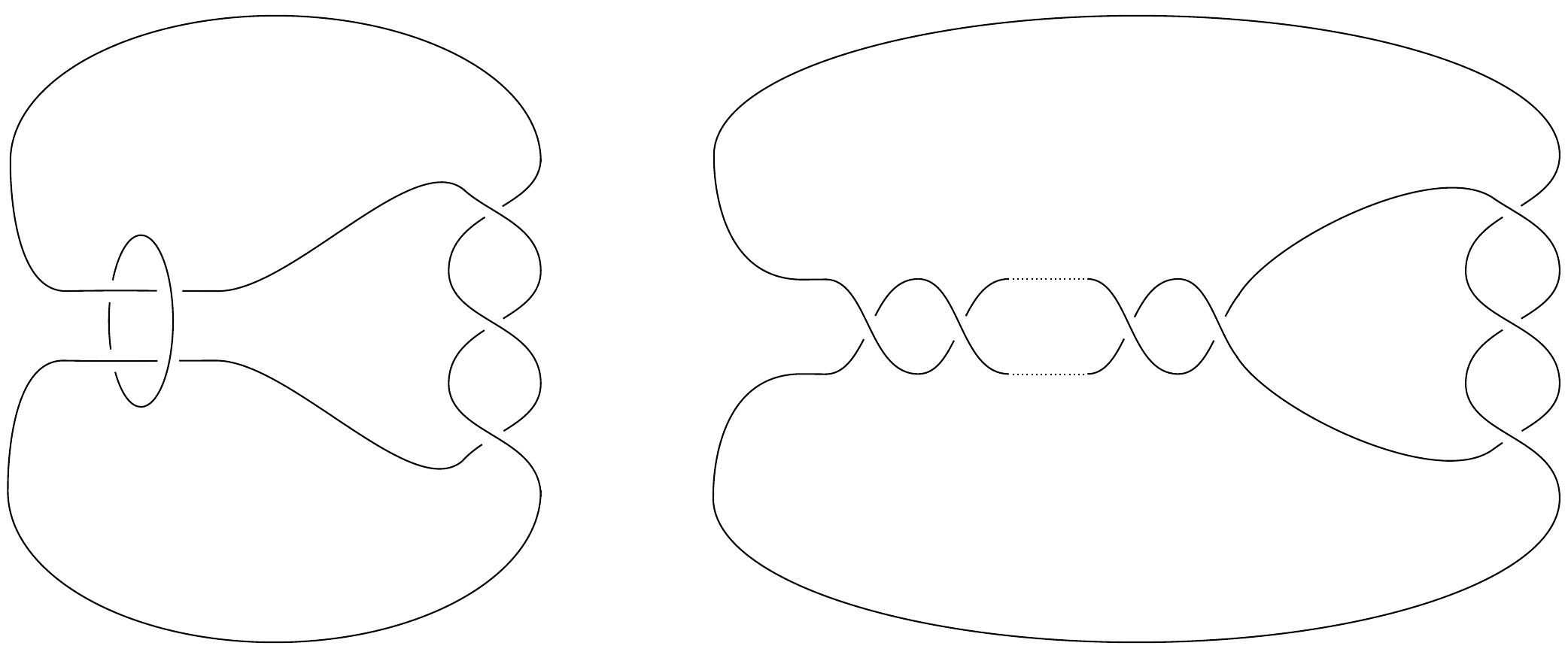}
\put(-340,-30){(a)}
\put(-115,-30){(b)}
\put(-370,52){$-\frac{1}{n}$}
\put(-405,65){$1$}
\put(-35,158){$4n+1$}
\put(-177,68){$\underbrace{\phantom{iiiiiiiiiiiiiiiiiiiiii}} $}
\put(-160,50){\tiny $2n$ right-handed}
\put(-150,40){\tiny half twists.}
\caption{Applying a Rolfsen twist to the $-\frac{1}{n}$-framed component of the link in Figure (a), yields the framed knot in Figure (b).}  \label{pic2}
\end{figure}
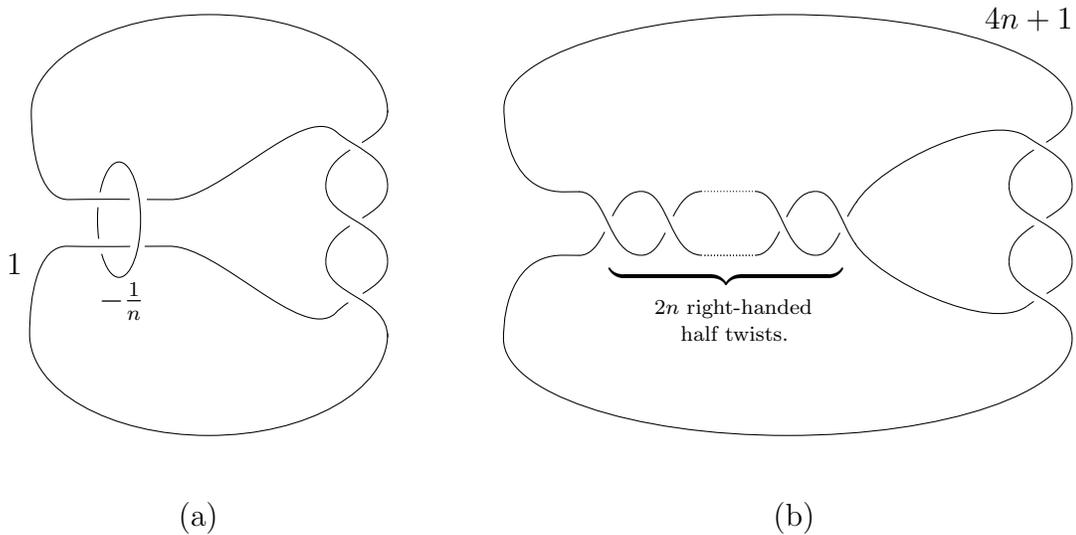
%
%
%
\subsection{Computations for Example \ref{ExampleOfNonLMinimizingKnots}} 
Let $K_{2m,2k+1}$ be the knot as defined by Figure \ref{pic3}, with $m,k\in \mathbb N$. It is easy to establish that $K_{2m,2k+1}$ is an alternating knot, with signature $2m$, and with Alexander polynomial 
$$\Delta_{2m,2k+1}(t) = (k+1)(t^m+t^{-m}) -(2k+1) \sum _{i=0}^{2m-2} (-1)^i t^{m-1-i},$$
from which $g(K_{2m,2k+1})=m$ follows. Together, these determine the knot Floer homology of $K_{2m,2k+1}$: 
$$\widehat{HFK}(K_{2m,2k+1},j) \cong \left\{
\begin{array}{cl}
\phantom{\mathbb Z^k_{(j+m)}\oplus}\mathbb Z_{(j+m)}\oplus \mathbb Z^{k}_{(j+m)} & \quad ; \quad j=m, \cr & \cr 
 \mathbb Z^{k}_{(j+m)}\oplus \mathbb Z_{(j+m)}\oplus \mathbb Z^{k}_{(j+m)} & \quad ; \quad |j|<m, \cr & \cr 
 \mathbb Z^k_{(j+m)}\oplus \mathbb Z_{(j+m)}\phantom{\oplus \mathbb Z^k_{(j+m)}} & \quad ; \quad j=-m
\end{array}
\right.$$
The vertical components of the $d_2$ differentials are then:

\centerline{
\xymatrix{
&  \mathbb Z_{(2m)} \ar[d]^\cong & \mathbb Z_{(2m)}^k \ar[d]^\cong\\
\mathbb Z_{(2m-1)}^k  \ar[d]^\cong&  \mathbb Z_{(2m-1)}  & \mathbb Z_{(2m-1)}^k \\
\mathbb Z_{(2m-2)}^k &  \mathbb Z_{(2m-2)}  \ar[d]^\cong & \mathbb Z_{(2m-2)}^k  \ar[d]^\cong \\
\mathbb Z_{(2m-3)}^k \ar@{.}[d]&  \mathbb Z_{(2m-3)}  \ar@{.}[d] & \mathbb Z_{(2m-3)}^k \ar@{.}[d] \\
\mathbb Z_{(3)}^k  \ar[d]^\cong &  \mathbb Z_{(3)}  & \mathbb Z_{(3)}^k \\
\mathbb Z_{(2)}^k &  \mathbb Z_{(2)} \ar[d]^\cong & \mathbb Z_{(2)}^k \ar[d]^\cong \\
\mathbb Z_{(1)}^k \ar[d]^\cong&  \mathbb Z_{(1)}  & \mathbb Z_{(1)}^k \\
\mathbb Z_{(0)}^k &  \mathbb Z_{(0)}  & 
}
}
\vskip1mm
The placement of the horizontal components of the $d_2$ differential in relation to the vertical ones, is a slightly more delicate task. Nevertheless, this placement is uniquely determined by the underlying algebra. This is evident for the horizontal differential acting on $\widehat{HFK}(K_{2m,2k+1},-m)\otimes U^t$, $t\in \mathbb Z$, and it is as in the large diagram on the next page. Once this differential is understood, it pins down uniquely the differential on $\widehat{HFK}(K_{2m,2k+1},-m+1)\otimes U^t$. Proceeding by induction, one obtained all the horizontal $d_2$ differentials. This procedure leads to:

\centerline{\tiny
\xymatrix{
& & & & &  \mathbb Z_{(2m+2)} \ar[d]^\cong & \mathbb Z_{(2m+2)}^k \ar[d]^\cong \\
&  \mathbb Z_{(2m)} \ar[d]_\cong & \mathbb Z_{(2m)}^k \ar[d]_\cong & & \mathbb Z_{(2m+1)}^k \ar@/_/[ll] \ar[d]^\cong&  \mathbb Z_{(2m+1)}  & \mathbb Z_{(2m+1)}^k\\
\mathbb Z_{(2m-1)}^k  \ar[d]_\cong&  \mathbb Z_{(2m-1)}  & \mathbb Z_{(2m-1)}^k & & \mathbb Z_{(2m)}^k \ar@/_/[ll] &  \mathbb Z_{(2m)}  \ar[d]^\cong  \ar@/_2pc/[llll] & \mathbb Z_{(2m)}^k  \ar[d]^\cong \\
\mathbb Z_{(2m-2)}^k &  \mathbb Z_{(2m-2)}  \ar[d]_\cong & \mathbb Z_{(2m-2)}^k  \ar[d]_\cong &  & \mathbb Z_{(2m-1)}^k  \ar@/_/[ll]  \ar[d]^\cong &  \mathbb Z_{(2m-1)}  & \mathbb Z_{(2m-1)}^k & \\
\mathbb Z_{(2m-3)}^k \ar[d]_\cong&  \mathbb Z_{(2m-3)}  & \mathbb Z_{(2m-3)}^k  & & \mathbb Z_{(2m-2)}^k   \ar@/_/[ll] \ar@{.}[d] &  \mathbb Z_{(2m-2)}  \ar@/_2pc/[llll]  \ar@{.}[d] & \mathbb Z_{(2m-2)}^k  \ar@{.}[d]\\
\mathbb Z_{(2m-4)}^k \ar@{.}[d]&  \mathbb Z_{(2m-4)}  \ar@{.}[d] & \mathbb Z_{(2m-4)}^k \ar@{.}[d]  & \vdots & \mathbb Z_{(6)}^k   &  \mathbb Z_{(6)} \ar[d]^\cong & \mathbb Z_{(6)}^k \ar[d]^\cong\\
\mathbb Z_{(4)}^k  &  \mathbb Z_{(4)}   \ar[d]_\cong & \mathbb Z_{(4)}^k \ar[d]_\cong & & \mathbb Z_{(5)}^k  \ar@/_/[ll] \ar[d]^\cong &  \mathbb Z_{(5)} & \mathbb Z_{(5)}^k  \\
\mathbb Z_{(3)}^k  \ar[d]_\cong &  \mathbb Z_{(3)}  & \mathbb Z_{(3)}^k & & \mathbb Z_{(4)}^k \ar@/_/[ll]&  \mathbb Z_{(4)} \ar@/_2pc/[llll]\ar[d]^\cong & \mathbb Z_{(4)}^k \ar[d]^\cong \\
\mathbb Z_{(2)}^k &  \mathbb Z_{(2)} \ar[d]_\cong & \mathbb Z_{(2)}^k \ar[d]_\cong & & \mathbb Z_{(3)}^k \ar@/_/[ll] \ar[d]^\cong&  \mathbb Z_{(3)}  & \mathbb Z_{(3)}^k\\
\mathbb Z_{(1)}^k \ar[d]_\cong&  \mathbb Z_{(1)}  & \mathbb Z_{(1)}^k  & \hspace{1cm}  &\mathbb Z_{(2)}^k \ar@/_/[ll]&  \mathbb Z_{(2)} \ar@/_2pc/[llll] &   \\
\mathbb Z_{(0)}^k &  \mathbb Z_{(0)} &   & & & & 
}
}
\vskip1mm
The homology of the various $\hat A_s$, is now easy to determine. Indeed, the only relevant part of $\hat A_s$ that needs examining, is the contribution from generators $[x,i,j]$ with $(i,j)\in \{ (0,s), (-1,s), (0,s-1)\}$. The homology depends on the parity of $s$. Firstly, if $|s|<m$ and $m-s$ is even, the relevant part looks like:
\vskip8mm
\centerline{\tiny
\xymatrix{
 \mathbb Z_{(2m-s-1)}^k &  \mathbb Z_{(2m-s-1)}  & \mathbb Z_{(2m-s-1)}^k& &  \mathbb Z_{(2m-s)}^k \ar@/_/[ll] &  \mathbb Z_{(2m-s)} \ar@/_2pc/[llll]  \ar[d]^\cong& \mathbb Z_{(2m-s)}^k \ar[d]^\cong  \\
&&& &  \mathbb Z_{(2m-s-1)}^k  \ar[d]^\cong &  \mathbb Z_{(2m-s-1)}  & \mathbb Z_{(2m-s-1)}^k  \\
&&&   & \mathbb Z_{(2m-3)}^k&  \mathbb Z_{(2m-3)}  & \mathbb Z_{(2m-3)}^k
}
}
\vskip1mm
Thus, $H_*(\hat A_s) \cong \mathbb Z\oplus \mathbb Z \oplus \mathbb Z$ and $\hat v_s$ and $\hat h_s$ are the projections $\pi_1$ and $\pi_2$ onto the first and second $\mathbb Z$-summand respectively. 

If $|s|<m$ and $m-s$ is odd, we get instead 
\vskip8mm
\centerline{\tiny
\xymatrix{
 \mathbb Z_{(2m-s-1)}^k &  \mathbb Z_{(2m-s-1)}  & \mathbb Z_{(2m-s-1)}^k& &  \mathbb Z_{(2m-s)}^k \ar@/_/[ll] \ar[d]^\cong &  \mathbb Z_{(2m-s)} & \mathbb Z_{(2m-s)}^k   \\
&&& &  \mathbb Z_{(2m-s-1)}^k  &  \mathbb Z_{(2m-s-1)}  \ar[d]^\cong  & \mathbb Z_{(2m-s-1)}^k  \ar[d]^\cong  \\
&&&   & \mathbb Z_{(2m-3)}^k&  \mathbb Z_{(2m-3)}  & \mathbb Z_{(2m-3)}^k
}
}
\vskip1mm
The homology of $\hat A_s$ is then isomorphic to $\mathbb Z \oplus \mathbb Z \oplus \mathbb Z^{2k+1}$ and $\hat v_s$ and $\hat h_s$ are again the projections $\pi_1$ and $\pi_2$ onto the first and second $\mathbb Z$-summand respectively.

Assuming that $p-(2m-1)q>0$, we proceed to compute the number of $L$-structures of $S^3_{-p/q}(K_{2m,2k+1})$ by dividing its spin$^c$-structures into two distinct groups:
\begin{itemize}
\item Those $[i]\in Spin^c(S^3_{-p/q}(K_{2m,2k+1}))$ for which $\lfloor \frac{i+ps}{q}\rfloor\le -m$ or $\lfloor \frac{i+ps}{q}\rfloor\ge m$ for all $s\in \mathbb Z$. 
\item  Those $[i]\in Spin^c(S^3_{-p/q}(K_{2m,2k+1}))$ for which there exists a unique $s_i\in \mathbb Z$ with $-g<\lfloor \frac{i+ps_i}{q}\rfloor<g$. 
\end{itemize}
Spin$^c$-structures of the first kind are guaranteed to be $L$-structures and there are $p-(2m-1)q$ of them. For the spin$^c$-structures of the second kind, we need to distinguish between those $s'_i=\lfloor \frac{i+ps_i}{q}\rfloor$ with $m-s_i'$ even and odd. 

If $m-s_i'$ is even, they our computation of $H_*(\hat A_{s'_i})$, $\hat v_{s_i'}$ and $\hat h_{s_i'}$ above  show that $\widehat{HF}(S^3_{-p/q}(K_{2m,2k+1}),[i])$ is isomorphic to the homology of the mapping cone:

\centerline{\tiny
\xymatrix{
& 
\mathbb Z  \ar[dl] & 
\mathbb Z  \ar[dl] & 
\mathbb Z\oplus \mathbb Z\oplus \mathbb Z \ar[dl]_{\pi_2} \ar[d]^{\pi_1}& 
\mathbb Z  \ar[d] &
\mathbb Z  \ar[d]   \\
\mathbb Z  & 
\mathbb Z  & 
\mathbb Z & 
\mathbb Z & 
\mathbb Z & 
\mathbb Z  \\ 
}
}
\vskip1mm
\noindent Its homology is easily found to be $\mathbb Z$ and so such spin$^c$-structures are $L$-structures. 

If $m-s_i'$ is odd, our computations of $H_*(\hat A_{s_i'})$, $\hat v_{s_i'}$ and $\hat h_{s_i'}$   show that $\widehat{HF}(S^3_{-p/q}(K_{2m,2k+1}),[i])$ is isomorphic to the homology of the mapping cone

\centerline{\tiny
\xymatrix{
& 
\mathbb Z  \ar[dl] & 
\mathbb Z  \ar[dl] & 
\mathbb Z\oplus \mathbb Z\oplus \mathbb Z^{2k+1} \ar[dl]_{\pi_2} \ar[d]^{\pi_1}& 
\mathbb Z  \ar[d] &
\mathbb Z  \ar[d]   \\
\mathbb Z  & 
\mathbb Z  & 
\mathbb Z & 
\mathbb Z & 
\mathbb Z & 
\mathbb Z  \\ 
}
}
\vskip1mm

which is isomorphic to $\mathbb Z^{2k+1}$. Thus, this $[i]$ is not an $L$-structure for any choice of $k\in \mathbb N$. 

To summarize, the number of $L$-structures coming from spin$^c$-structures of the second kind is $(m-1)q$, which when added to the number $p-(2m-1)q$ of $L$-structures from spin$^c$-structures of the first kind, gives a total of $p-mq$ $L$-structures on $S^3_{-p/q}(K_{2m,2k+1})$, as claimed in Example \ref{ExampleOfNonLMinimizingKnots}.

To complete proving the claims made in Example \ref{ExampleOfNonLMinimizingKnots}, we need to demonstrate that the framed knots $(K_{2m,2k+1},-\frac{4mn-1}{n})$ and $(K_{2n,2k+1},-\frac{4mn-1}{m})$ are surgery equivalent. This is accomplished by a sequence of Rolfsen twists as explained in Figure \ref{pic5}. 
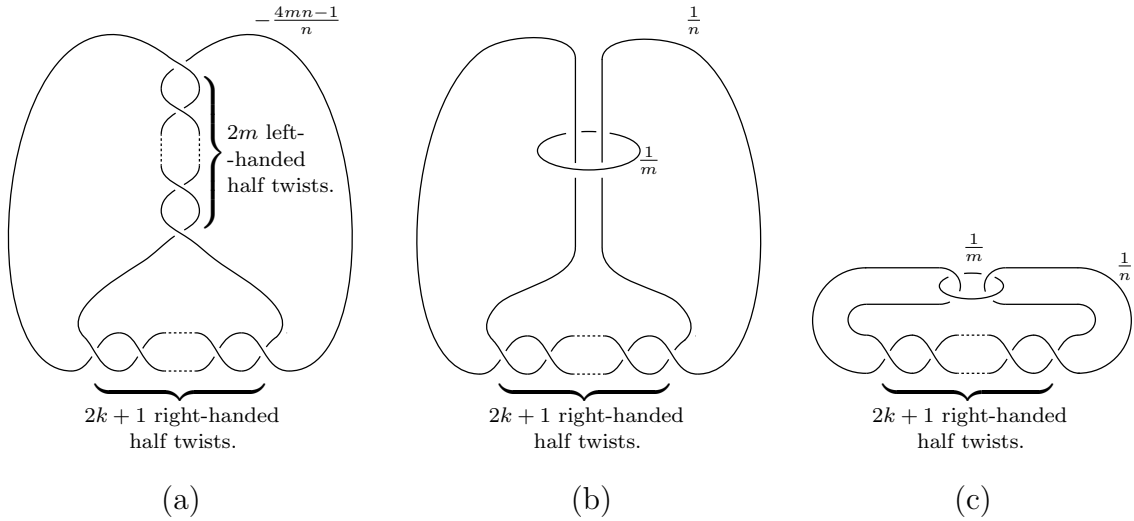
\begin{figure}[htb!] 
\centering
\includegraphics[width=15cm]{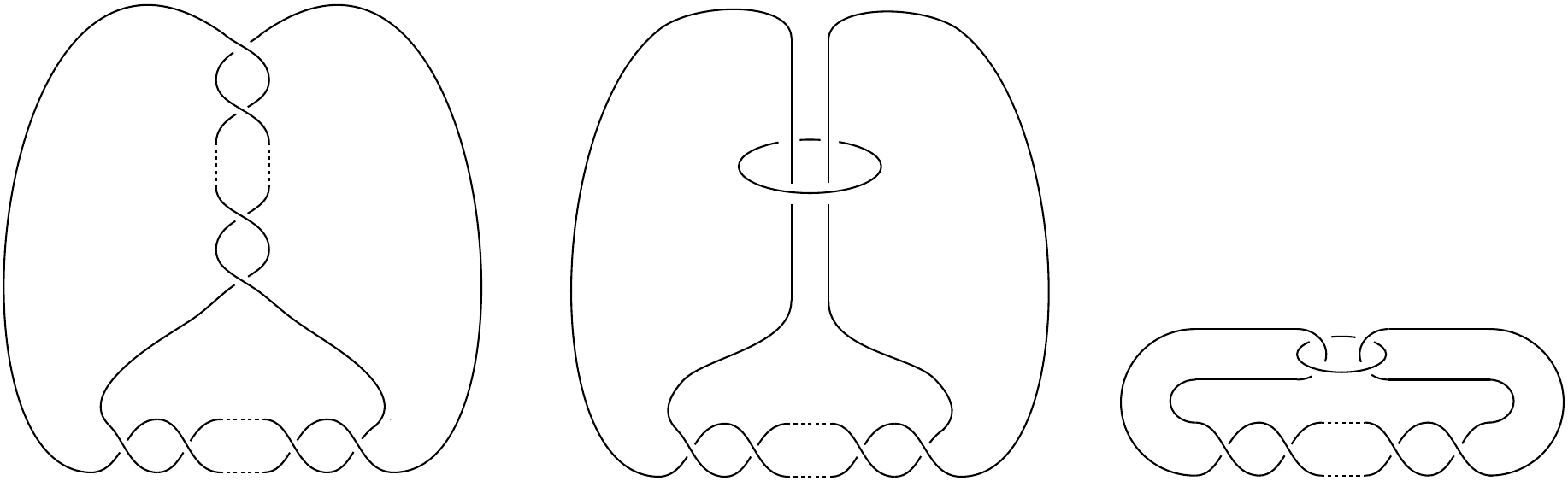}
\put(-393,-2){$\underbrace{\phantom{iiiiiiiiiiiiiiii}}$}
\put(-240,-2){$\underbrace{\phantom{iiiiiiiiiiiiiiii}}$}
\put(-95,-2){$\underbrace{\phantom{iiiiiiiiiiiiiiii}}$}
\put(-397,-17){\tiny $2k+1$ right-handed}
\put(-380,-27){\tiny half twists.}
\put(-244,-17){\tiny $2k+1$ right-handed}
\put(-227,-27){\tiny half twists.}
\put(-99,-17){\tiny $2k+1$ right-handed}
\put(-82,-27){\tiny half twists.}
\put(-365,82){$\left. \begin{array}{c} \cr \cr \cr \cr \end{array}\right\}$}
\put(-343,90){\tiny $2m$ left-}
\put(-343,80){\tiny -handed}
\put(-343,70){\tiny half twists.}
\put(-333,133){\tiny $-\frac{4mn-1}{n}$}
\put(-188,80){\tiny $\frac{1}{m}$}
\put(-170,133){\tiny $\frac{1}{n}$}
\put(-65,47){\tiny $\frac{1}{m}$}
\put(-7,40){\tiny $\frac{1}{n}$}
\put(-368,-50){(a)}
\put(-213,-50){(b)}
\put(-68,-50){(c)}
\caption{The knot $K_{2m,2k+1}$ with framing $-\frac{4mn-1}{n}$ in picture (a) is obtained from the framed link in picture (b) by performing a Rolfsen twist along the unknot with framing $1/m$. The framed link in picture (c), gotten by an isotopy from the link in picture (b), is symmetric under interchanging $m$ and $n$. Accordingly, the framed knots $\left(K_{2m,2k+1},-\frac{4mn-1}{n} \right)$ and $\left(K_{2n,2k+1},-\frac{4mn-1}{m} \right)$ are sugery equivalent. }  \label{pic5}
\end{figure}

\end{document}